\documentclass[12pt]{article}
\usepackage{amsmath,amssymb,amsthm}
\usepackage{array}
\usepackage{fullpage}
\usepackage{enumitem}
\usepackage[table]{xcolor}
\usepackage{graphicx}
\usepackage[colorlinks=true,
linkcolor=webgreen,
filecolor=webbrown,
citecolor=webgreen]{hyperref}
\definecolor{webgreen}{rgb}{0,.5,0}
\definecolor{webbrown}{rgb}{.6,0,0}

\newcommand{\seqnum}[1]{\href{https://oeis.org/#1}{\rm \underline{#1}}}

\theoremstyle{plain}
\newtheorem{theorem}{Theorem}
\newtheorem{corollary}[theorem]{Corollary}
\newtheorem{lemma}[theorem]{Lemma}
\newtheorem{proposition}[theorem]{Proposition}

\theoremstyle{definition}
\newtheorem{definition}[theorem]{Definition}
\newtheorem{example}[theorem]{Example}

\theoremstyle{remark}
\newtheorem{remark}[theorem]{Remark}

\title{A $3$-adic Recurrence for the Fixed Points of the Josephus Function $J_4$}
\author{
Yunier Bello-Cruz\\
\small School of Mathematical and Statistical Sciences\\
\small Northern Illinois University\\
\small DeKalb, IL 60115\\
\small USA\\
\small \texttt{yunierbello@niu.edu}
\and
Roy Quintero-Contreras\\
\small School of Mathematical and Statistical Sciences\\
\small Northern Illinois University\\
\small DeKalb, IL 60115\\
\small USA\\
\small \texttt{rquinterocontreras@niu.edu}
}
\date{}

\begin{document}
\maketitle

\begin{abstract}
\noindent In the Josephus problem with stepsize four, the participants in a
circle are eliminated one by one, every fourth person leaving, until a single
survivor remains. A fixed point occurs when the survivor turns out to be the
person who began in the last seat. The circle sizes with this property form
the sequence 1; 21; 38; 51; 122; 163; 689; 919; 2{,}906; and so on, whose gaps
fluctuate erratically. This paper explains the fluctuation and turns it into
a recurrence. Between consecutive fixed points, the circle sizes at which the
survivor falls exactly one or two seats short of the last one, the
near-misses, group into alternating blocks of the two kinds, and the length
of every block is the number of times three divides a simple quantity built
from the circle size that precedes the block. Iterating these divisibility
counts carries each fixed point to the next. Stepsize four is the first case
in which two kinds of near-miss coexist, and the alternation they force is
what separates it from the solved cases of stepsizes two and three. As a
byproduct, the survivor's position for an arbitrary circle size can be
computed by walking the near-misses of a single interval, in a number of
steps proportional to their count, rather than stepping through every smaller
circle as the defining recursion does.
\end{abstract}

\section{Introduction}\label{sec:introduction}

In the Josephus problem, $n$ participants stand in a circle and are eliminated in
turn, every $k$-th at a time, until a single survivor remains. The position of
that survivor, viewed as a function of $n$, is denoted $J_k(n)$, where $J_k$ is the Josephus function. The
problem originates in antiquity \cite{Bac,Jos} and has been revisited many times
since. Euler
gave an early recursive treatment~\cite{Eul}, and the case $k=2$ admits the familiar
closed form $J_2(n)=2n-2^{\lfloor\log_2 n\rfloor+1}+1$ recorded in~\cite{Gra,Knu}.
For general $k$ the function is far less transparent, and the literature has
approached it from two main directions.

One direction is algorithmic. Knuth describes two $O(n\log n)$ schemes for generating
the elimination order~\cite{Gra,Knu}, and Lloyd later gave an $O(n\log m)$
algorithm~\cite{Llo}; the survivor itself can be computed in linear time from the
standard reduction $J_k(n)=\bigl(J_k(n-1)+k-1\bigr)\bmod n+1$. The other direction
studies the structure of $J_k$ and its variants. Robinson~\cite{Rob} and
Halbeisen-Hungerb\"uhler~\cite{Hal} analyze the asymptotic behavior of the survivor,
Odlyzko and Wilf relate the problem to functional iteration~\cite{OW}, and several
authors have introduced generalizations: Tait~\cite{Tai}, Uchiyama~\cite{Uch}, and
Th\'eriault~\cite{The} treat variants with different counting rules, while Ruskey and
Williams~\cite{RW} consider a version in which each participant has several lives. The
relevant integer sequences are catalogued in the OEIS~\cite{OEIS}. The present paper
belongs to the second, structural, line of work, and continues the study of the fixed
points of $J_k$ begun in~\cite{BCQC-JIS,BCQC-J3,BCQC-FBE}.

A positive integer $n$ is a fixed point of $J_k$ if $J_k(n)=n$: the survivor
ends up in the very seat they started from, no matter how the counting falls. Fixed
points exist for every $k\ge2$ and form a strictly increasing sequence
$\bigl(n_p^{(\ell)}\bigr)_{\ell\ge1}$; see~\cite{BCQC-JIS}. The question we take up is
how one fixed point determines the next. An answer is worth having because a
recurrence on the fixed points lets one evaluate $J_k$ at large arguments without
marching through all the intermediate values, as we show for $k=4$ in
Subsection~\ref{sub:evaluation}.

A picture helps. Figure~\ref{fig:J4_1_80} plots $J_4$: the values climb in straight
runs of slope~$4$, each run topping out just under the diagonal $y=n$ before dropping
back down. The tops of those runs are the points that matter. Most of them fall one or
two short of the diagonal; a few land exactly on it, and those are the fixed points.
Take the fixed point $n_p^{(2)}=21$. The next run tops out at $n=28$ with
$J_4(28)=27$, one short of the diagonal, and then a single long climb reaches the next
fixed point $n_p^{(3)}=38$. That lone near-miss at $28$ is the entire story of the
interval $(21,38)$, and the goal of this paper is to predict such near-misses, how
many, of what kind, and where, directly from $21$.

Those run-tops are what the literature calls high extremal points: the sizes $n_e$
at which $J_k$ attains a local maximum relative to the diagonal $y=n$. A high
extremal point that is not itself a fixed point (a near-miss, in the language
above) is called pure. How the pure points are distributed inside each interval
$(n_p^{(\ell)},n_p^{(\ell+1)})$ is exactly what controls the jump to the next fixed
point, and the difficulty of describing them grows with $k$.

For $k=2$ there is nothing to describe: the fixed points are simply
$n_p^{(\ell)}=2^\ell-1$, and no pure points lie between them. The case $k=3$ is the
first with pure points, and they are still tame: in each interval they are all of one
kind, and their number is the single index
\[
\overline{m}_\ell := \nu_2\!\bigl(3n_p^{(\ell)}+2\bigr),
\]
the $2$-adic valuation of $3n_p^{(\ell)}+2$. This one index drives the explicit recurrence for the terms of the fixed points sequence of $J_3$
(\seqnum{A182459} in the OEIS~\cite{OEIS}),
\[
n_p^{(\ell+1)}
=\frac{3^{\overline{m}_\ell}(3n_p^{(\ell)}+2)-2^{\overline{m}_\ell}}{2^{\overline{m}_\ell+1}},
\]
established in~\cite{BCQC-J3}, which passes directly from $n_p^{(\ell)}$ to
$n_p^{(\ell+1)}$ without computing any intermediate value. The arithmetic of this
sequence is examined further in~\cite{BCQC-FBE}, where a connection with the Chinese
Remainder Theorem is drawn and a recursive digit pattern in the base-$3/2$ expansions
of consecutive fixed points is identified.

This paper carries out the corresponding analysis for $k=4$, where the picture
changes in one decisive way: the pure points come in two kinds, not one. A pure point
$n_e$ is of type-$1$ if $J_4(n_e)=n_e-1$ and of type-$2$ if $J_4(n_e)=n_e-2$:
it misses the diagonal by one or by two. Inside each interval the two kinds do not
intermingle freely; they group into blocks, where a block is a maximal run of pure
points of a single type, and consecutive blocks alternate in type. The interval
$(21,38)$ above is the simplest case: one block, holding the single type-$1$ point
$28$.

The analysis proceeds in three stages. First, the residue of $n_p^{(\ell)}$ modulo~$3$, which we call its ternarity, fixes the type of the first block, and a single
$3$-adic valuation, $\nu_3(4n_p^{(\ell)}+3)$ or $\nu_3(4n_p^{(\ell)}+2)$ according to
that type, fixes the block's length and its last point. Second, given the last point
$n_e$ of one block, the length of the next is delivered by one of two
integer-valued functions, the transition meters $\mu_{1\to2}$ and
$\mu_{2\to1}$; these too turn out to be single $3$-adic valuations,
$\nu_3(4n_e+3)$ and $\nu_3(4n_e+5)$ (Corollary~\ref{C:meters-val}), though
proving it takes a floor-by-floor divisibility analysis. Third, a termination rule
recognizes when no further block is possible and turns the last pure point directly
into $n_p^{(\ell+1)}$. The main theorem (Theorem~\ref{thm:main}) strings these three
stages together into a recurrence from $n_p^{(\ell)}$ to $n_p^{(\ell+1)}$.

The contrast with $k=3$ is concrete. There a single valuation index settles an
entire interval. For $k=4$ each interval needs a chain of valuations, one per
block, alternating between the two linear forms above, and establishing each
link costs a floor-by-floor test built on two offset sequences, $(A_t)_{t\ge0}$
and $(B_t)_{t\ge1}$, introduced in Section~\ref{sec:meters}. The chain is the
extra arithmetic that having two kinds of pure points demands.

Several integer sequences arise naturally in this study. The Josephus survivor
function itself appears in the OEIS~\cite{OEIS} as the triangle \seqnum{A032434}, with
the case $k=2$ recorded separately as \seqnum{A006257}. The central object here is
the fixed point sequence of $J_4$,
\[
1;\,21;\,38;\,51;\,122;\,163;\,689;\,919;\,2{,}906;\,3{,}875;\,5{,}167;\,51{,}617;\,68{,}823;\,
163{,}137;\,290{,}022; \dots\, ;
\]
which we denote $\bigl(n_p^{(\ell)}\bigr)_{\ell\ge1}$ and study throughout
(\seqnum{A385333}). Two further sequences derived from it record the combinatorial
shape of each interval from the second fixed point $n_p^{(2)}=21$ on; the interval
after $n_p^{(1)}=1$ lies in the small irregular range $n<6$ excluded by the
structural theorem and is omitted. The first records the number of blocks $r(\ell)$
between $n_p^{(\ell)}$ and $n_p^{(\ell+1)}$ for $\ell\ge2$,
\[
1;\,0;\,1;\,0;\,3;\,0;\,2;\,0;\,0;\,5;\,0;\,2;\,1;\,4;\,1;\,0;\,1;\,1;\,0;\,2;\,0;\,0;\,0;\,1;\dots;
\]
and the second records the type $t_1(\ell)$ of the first block in each such interval
($0$ for a direct transition with no pure points, otherwise $1$ or $2$),
\[
1;\,0;\,1;\,0;\,2;\,0;\,2;\,0;\,0;\,2;\,0;\,1;\,1;\,1;\,1;\,0;\,1;\,1;\,0;\,2;\,0;\,0;\,0;\,2;\dots .
\]
The recurrence developed in Section~\ref{sec:structure} computes all of these without
enumerating the intervening values of $J_4$.

The paper is organized as follows.
Section~\ref{sec:definitions} sets up the definitions and notation: high extremal
points, ternarity, and the $3$-adic valuation.
Section~\ref{sec:blocks} describes the block structure between consecutive fixed
points.
Section~\ref{sec:extremal} proves the first-block formula and recalls
from~\cite{BCQC-JIS} the local linearity of $J_4$ on which everything rests.
Section~\ref{sec:meters} introduces the transition meters and their divisibility
diagrams, and derives their closed valuation form.
Section~\ref{sec:structure} ties everything together: it proves the termination rule
(Subsection~\ref{sub:arrival}), assembles the rules into the main structural theorem
(Subsection~\ref{sub:main}), and uses that theorem to evaluate $J_4(n)$ at an
arbitrary argument without iterating through intermediate values
(Subsection~\ref{sub:evaluation}).
\section{Definitions and notation}\label{sec:definitions}

This section fixes the notation and records the type-value identity, the only fact needed at this stage. The local linearity of $J_4$, which drives the rest of the analysis, is recalled in Section~4.
Throughout, $\mathbb{N}=\{1,2,3,\dots\}$ denotes the positive integers.

\begin{definition}[Josephus function and fixed points]\label{D:josephus}
For $n\in\mathbb{N}$, the value $J_4(n)$ is the original position of the
survivor when $n$ participants stand at positions $1,\dots,n$ around a circle,
the count starts at position $1$, and every fourth participant is removed.
Equivalently, $J_4(1)=1$ and
\[
J_4(n)=\bigl(J_4(n-1)+3\bigr)\bmod n+1 \qquad (n\ge2).
\]
A positive integer $n_p$ is a fixed point of $J_4$ if $J_4(n_p)=n_p$: the
participant in the last seat survives. The fixed points form a strictly
increasing sequence, denoted $\bigl(n_p^{(\ell)}\bigr)_{\ell\ge1}$ with
$n_p^{(1)}=1$.
\end{definition}

The next two definitions name the two coordinates we read off a fixed point: its
residue modulo~$3$, which decides the type of the first block (Section~\ref{sec:blocks}),
and the local shape of $J_4$ near the diagonal, which produces the blocks in the first
place.

\begin{definition}[Ternarity]\label{D:ternarity}
The ternarity of $n\in\mathbb{N}$ is its residue $n \bmod 3 \in\{0,\,1,\,2\}$.
\end{definition}

As $J_4$ runs, its graph repeatedly climbs almost up to the diagonal $y=n$ and then
drops back; the points where it comes closest are the ones that carry all the
structure.

\begin{definition}[High extremal points, types, and pure points]\label{D:hep}
A positive integer $n_e$ is a high extremal point of $J_4$ if
\[
J_4(n_e)\in\{n_e;\,n_e-1;\,n_e-2\}.
\]
Its type is $t:=n_e-J_4(n_e)\in\{0,1,2\}$, the amount by which it falls short of the
diagonal. A high extremal point of type-$0$ is a fixed point; one of type-$1$ or
type-$2$ is called a pure point.
\end{definition}

The displayed condition is meaningful for every $n_e$, but we use it only for
$n_e\ge6$. The first five values of $J_4$ are $1;\, 1;\, 2;\, 2;\, 1$;  so $n=2,\, 3$, and
$4$ satisfy the condition without sitting at the top of any slope-$4$ climb;
from $n=6$ on, the high extremal points are exactly the run-tops visible in
Figure~\ref{fig:J4_1_80}, and every statement below carries the threshold
$n_e\ge6$.

The whole analysis rests on one local fact about $J_4$, which we isolate here because
it is invoked at nearly every step.

\begin{lemma}[Type-value identity]\label{L:type-value}
For every high extremal point $n_e\ge 6$,
\[
J_4(n_e+1)=3-t,
\qquad\text{where } t \text{ is the type of } n_e.
\]
\end{lemma}

\begin{proof}
This is~\cite[Theorem~5(b)]{BCQC-JIS}.
\end{proof}

In words: the value of $J_4$ at the integer just past a high extremal point reads off
its type, $2$ after a type-$1$ point, $1$ after a type-$2$ point, and $3$ after a
fixed point. The threshold $n_e\ge6$ excludes only the initial irregular values
$J_4(1),\dots,J_4(5)$, after which the structure settles into the regular pattern we
now describe. Calling $n$ a low
extremal point when $n\ge6$ and $J_4(n)\in\{1,2,3\}$, high and low extremal points
strictly alternate from $n=6$ on, and the low extremal point immediately after $n_e$
is $n_e+1$, which is exactly the integer the type-value identity describes.

The block lengths in Section~\ref{sec:meters} will be counted by powers of~$3$
dividing certain linear forms, so we also fix the standard valuation.

\begin{definition}[$3$-adic valuation]\label{D:valuation}
For $m\in\mathbb{N}$, the $3$-adic valuation $\nu_3(m)$ is the largest integer
$t\ge 0$ with $3^t\mid m$.
\end{definition}

Figure~\ref{fig:J4_1_80} shows $J_4(n)$ for $1\le n\le 165$, with its six fixed
points $1;\, 21;\, 38;\, 51;\, 122;\, 163$ on the diagonal and the pure points marked by type.
Two features visible there are used constantly. The graph rises in straight segments
of slope~$4$, so between two consecutive high extremal points $J_4$ is completely
determined by its starting value; and each segment ends at a high extremal point,
just short of (or on) the diagonal. These two facts are what let us replace the
step-by-step recurrence for $J_4$ by jumps from one high extremal point to the next.

\begin{figure}[!h]
\centering
\includegraphics[width=0.72\linewidth]{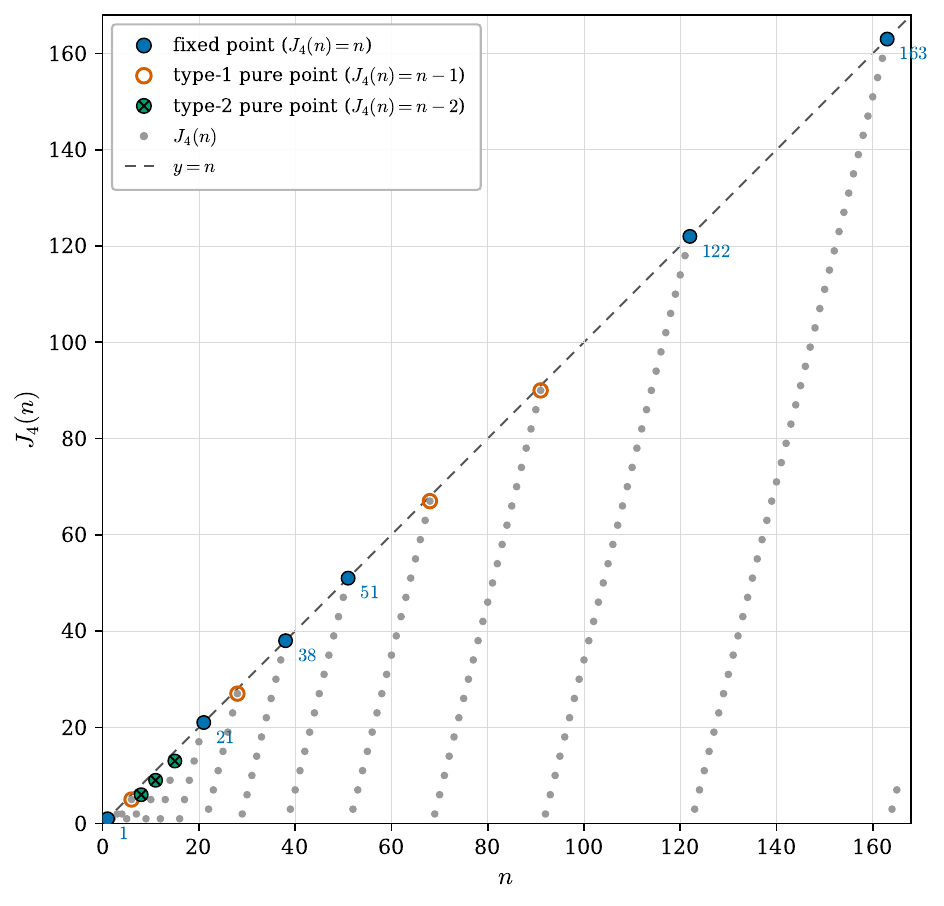}
\caption{The Josephus function $J_4(n)$ for $1\le n\le 165$, plotted as discrete
values. Each climb of slope~$4$ rises toward the diagonal $y=n$ and is capped by a
high extremal point. The fixed points shown are
$n=1; 21; 38; 51; 122; 163$; between them sit the pure high extremal points, separated by
type: type-$1$ points ($J_4(n)=n-1$) at $n=6;\,28;\,68;\,91$, and type-$2$ points
($J_4(n)=n-2$) at $n=8;\,11;\,15$. Pure points are marked only for $n\ge6$, the regular range in which the
type-value identity (Lemma~\ref{L:type-value}) applies. The interval $(51,\,122)$ shows a run of two type-$1$
points ($n=68,\,91$); the interval $(122,\,163)$ is a direct transition, with no pure
points between the two fixed points, as happens whenever the lower fixed point is
$\equiv 2$ (mod $3$) (Proposition~\ref{P:initial-block}).}
\label{fig:J4_1_80}
\end{figure}

Table~\ref{tab:J4_fixed_points} lists the first fifteen fixed points of $J_4$,
illustrating the rapid and irregular growth of the sequence.

\begin{table}[htbp]
\centering
\small
\setlength{\tabcolsep}{4pt}
\begin{tabular}{c|*{15}{c}}
\hline
$\ell$
& 1 & 2 & 3 & 4 & 5 & 6 & 7 & 8
& 9 & 10 & 11 & 12 & 13 & 14 & 15 \\
\hline
$n_p^{(\ell)}$
& 1 & 21 & 38 & 51 & 122 & 163 & 689 & 919
& 2{,}906 & 3{,}875 & 5{,}167 & 51{,}617 & 68{,}823 & 163{,}137 & 290{,}022 \\
\hline
\end{tabular}
\caption{The first fifteen fixed points of $J_4$.}
\label{tab:J4_fixed_points}
\end{table}

\begin{example}\label{Ex:919}
Take the fixed point $n_p^{(8)}=919$. In the interval up to the next fixed point
$n_p^{(9)}=2{,}906$ there are exactly three pure points: $1{,}225$ of type-$2$, then
$1{,}634$ and $2{,}179$ of type-$1$. Reading them in order, the types run
$2;\,1;\,1$: a single type-$2$ point followed by two type-$1$ points. This grouping into
runs of one type, here a length-$1$ block then a length-$2$ block, is the
pattern the next section makes precise.
\end{example}
\section{Block structure between consecutive fixed points}\label{sec:blocks}

Section~\ref{sec:definitions} described single pure points; here we describe how they
sit together. Inside one interval the pure points are not scattered at random: equal
types cluster, and the clusters alternate. We record this with the language of blocks,
which carries through the rest of the paper.

\begin{definition}[Blocks of pure points]\label{D:blocks}
Let $\bigl(n_p^{(\ell)},\,n_p^{(\ell+1)}\bigr)$ be the open interval between two
consecutive fixed points, and suppose it contains at least one pure point. List
the pure points of the interval in increasing order. A block is a maximal run
of points of the same type within that list: consecutive entries of the list,
not consecutive integers, so the type-$1$ points $290$ and $387$ of
Example~\ref{Ex:163_689} below form a single block. Listed in increasing order,
the blocks of the interval are
\[
\mathcal B_1;\,\mathcal B_2;\,\dots;\,\mathcal B_r;\,\qquad r\ge 1,
\]
and we write $n_e^{(\ell,j)}$ for the last pure point of block $\mathcal B_j$.
The block count $r$ depends on the interval; we write $r=r(\ell)$ when that
dependence matters. The calligraphic letter keeps the blocks apart from the
offset sequence $(B_t)$ of Section~\ref{sec:meters}.
\end{definition}

We single out the last point $n_e^{(\ell,j)}$ of each block because, as
Section~\ref{sec:meters} shows, it is the input from which the length of the next
block is computed: each block hands its successor a starting point, and the chain of
these points is what the recurrence follows.

\begin{remark}\label{R:alternation}
Whenever an interval contains pure points, they form a finite alternating run
of blocks, consecutive blocks having opposite types. Finiteness comes from
Lemma~\ref{L:monotone} and the alternation from Corollary~\ref{C:successor},
both proved in Section~\ref{sec:extremal}: a maximal run of one type ends at a
point $m\not\equiv2$ (mod $3$), so the high extremal point after $m$ is either
the next fixed point or a pure point of the opposite type. We use both facts
freely from here on.
\end{remark}

\begin{example}\label{Ex:163_689}
Between the consecutive fixed points $n_p^{(6)}=163$ and $n_p^{(7)}=689$ there are
three blocks ($r=3$):
\[
\mathcal B_1=\{217\};\qquad \mathcal B_2=\{290;\,387\};\qquad \mathcal B_3=\{516\}.
\]
Their types are $2,1,2$, alternating, as Remark~\ref{R:alternation} requires, 
and the last points of the blocks are
\[
n_e^{(6,1)}=217;\qquad n_e^{(6,2)}=387;\qquad n_e^{(6,3)}=516.
\]
\end{example}

\begin{remark}[Possible block patterns]\label{R:block-patterns}
An interval therefore shows one of two behaviors: a direct transition with no pure
points at all, or an alternating chain $\mathcal B_1,\dots,\mathcal B_r$. A single block ($r=1$) is of
one type; for $r\ge2$ the first and last blocks may each be of either type, so the
chain falls into one of four endpoint patterns,
\[
(1\to\cdots\to1),\quad
(1\to\cdots\to2),\quad
(2\to\cdots\to1),\quad
(2\to\cdots\to2),
\]
all of which occur, though not equally soon. The pattern $2\to\cdots\to1$
appears in Example~\ref{Ex:919} and $2\to\cdots\to2$ in
Example~\ref{Ex:163_689}; the two examples below show $1\to\cdots\to2$ and a
longer instance of $2\to\cdots\to2$. The pattern $1\to\cdots\to1$ is much
rarer: among the intervals between the first sixty fixed points it occurs
exactly once, in the last of them, which opens at the twenty-digit fixed point
$n_p^{(59)}=67{,}668{,}528{,}084{,}558{,}408{,}666$ and holds three singleton
blocks of types $1;\,2;\,1$ (checked with the recurrence of
Section~\ref{sec:structure}).
\end{remark}

\begin{example}[A four-block pattern]\label{Ex:pattern1212}
Here the four blocks between $n_p^{(15)}=290{,}022$ and $n_p^{(16)}=1{,}629{,}537$
realize the pattern $1\to2\to1\to2$:
\[
\begin{gathered}
\underbrace{\{386{,}696;\,515{,}595\}}_{1};\qquad
\underbrace{\{687{,}460\}}_{2};\qquad
\underbrace{\{916{,}614\}}_{1};\qquad
\underbrace{\{1{,}222{,}152\}}_{2},
\end{gathered}
\]
where the subscript on each block gives its type.
Here the first block has two points and the rest are singletons, so the last points
are $n_e^{(15,1)}=515{,}595$; $n_e^{(15,2)}=687{,}460$; $n_e^{(15,3)}=916{,}614$; and
$n_e^{(15,4)}=1{,}222{,}152$.
\end{example}

\begin{example}[A five-block pattern]\label{Ex:pattern21212}
Here the five blocks between $n_p^{(11)}=5{,}167$ and $n_p^{(12)}=51{,}617$ realize
the pattern $2\to1\to2\to1\to2$:
\[
\begin{gathered}
\underbrace{\{6{,}889\}}_{2};\qquad
\underbrace{\{9{,}186\}}_{1};\qquad
\underbrace{\{12{,}248;\,16{,}331;\,21{,}775\}}_{2};\qquad
\underbrace{\{29{,}034\}}_{1};\qquad
\underbrace{\{38{,}712\}}_{2},
\end{gathered}
\]
again with the subscript recording the type of each block.
The middle block has three points; this is the interval whose lengths we recompute
from scratch in Section~\ref{sec:meters} (Example~\ref{Ex:altmeters_6889}).
\end{example}

\section{Extremal point analysis}\label{sec:extremal}

The arguments of this paper all run on one mechanism: given a high extremal point,
produce the next one and read off its type. This section assembles that mechanism. The
first two lemmas specialize results of~\cite{BCQC-JIS} to $k=4$ and serve as
computational tools; the third records a growth bound that guarantees finiteness.
Proposition~\ref{P:initial-block}, the first original result of the paper, then uses
them to pin down the first block of pure points after a fixed point. The argument
follows the template used for $k=3$ in~\cite{BCQC-J3}.

\begin{lemma}[Local linearity and next high extremal point for $J_4$]
\label{L:J4_engine}
Let $n_e\ge 6$ be a high extremal point of $J_4$ and set
\[
j:=J_4(n_e+1)\in\{1,2,3\},\qquad
\varepsilon\equiv n_e \pmod 3,\quad \varepsilon\in\{0,1,2\}.
\]
Define
\[
\delta(\varepsilon,j):=
\begin{cases}
1, & \varepsilon< j-1,\\
0, & \varepsilon\ge j-1,
\end{cases}
\qquad
m_0:=\left\lfloor\frac{n_e-(j-1)}{3}\right\rfloor.
\]
Then the following hold.
\begin{enumerate}
\item [ {\rm (a)}]
\(
J_4(n_e+1+m)=j+4m,\) for all \(m\in\{0,1,\dots,m_0\}.
\)
\item [ {\rm (b)}] The next high extremal point after $n_e$ is
\[
n_e^{+}:=\frac{4(n_e+1)-(\varepsilon+1)}{3}-\delta(\varepsilon,j),
\]
and, moreover,
\[
J_4(n_e^{+})=n_e^{+}-\bigl(3\delta(\varepsilon,j)+\varepsilon-(j-1)\bigr),
\qquad
J_4(n_e^{+}+1)=3(1-\delta(\varepsilon,j))-(\varepsilon-(j-1)).
\]
The value $3\delta(\varepsilon,j)+\varepsilon-(j-1)\in\{0,1,2\}$ determines the type
of $n_e^+$, and $\delta(\varepsilon,j)$ accounts for the rounding that arises when
$\varepsilon<j-1$.
\item [ {\rm (c)}]
The point $n_e^{+}$ is a fixed point of $J_4$ if and only if $\varepsilon=j-1$, in
which case
\[
n_e^{+}=\frac{4(n_e+1)-j}{3}.
\]
\end{enumerate}
\end{lemma}

\begin{proof}
Everything here is the $k=4$ case of Theorem~5 of~\cite{BCQC-JIS}; we only rename
that paper's residue variable $r$ to $\varepsilon$.

\noindent {\rm (a)}
Taking $k=4$ in~\cite[Theorem~5(a)]{BCQC-JIS} gives
\[
J_4(n_e+1+m)=j+4m
\qquad\text{for all }m\in\bigl\{0,1,\dots,\lfloor(n_e-(j-1))/3\rfloor\bigr\},
\]
which is the claim with $m_0=\lfloor(n_e-(j-1))/3\rfloor$.

\noindent {\rm (b)}
Write $n_e=3q+\varepsilon$ with $q=\lfloor n_e/3\rfloor\ge2$, and abbreviate
$\delta:=\delta(\varepsilon,j)$. By~\cite[Lemma~4(b)]{BCQC-JIS} with $k=4$ we have
$m_0=q-\delta$, so the next high extremal point is
$n_e^+=n_e+1+m_0=4q+\varepsilon+1-\delta$. This is an integer: since
\[
4(n_e+1)-(\varepsilon+1)=4(3q+\varepsilon+1)-(\varepsilon+1)=12q+3\varepsilon+3
=3(4q+\varepsilon+1),
\]
the quotient $(4(n_e+1)-(\varepsilon+1))/3$ equals $4q+\varepsilon+1$, and
subtracting $\delta\in\{0,1\}$ keeps it integral.

The value of $J_4$ at $n_e^+$ comes from~\cite[equation~(10)]{BCQC-JIS} with $k=4$,
$r=\varepsilon$:
\[
J_4(n_e^+)=n_e^+-\delta(k-1)-\varepsilon+(j-1)=n_e^+-(3\delta+\varepsilon-(j-1)),
\]
and the value at $n_e^++1$ from~\cite[equation~(11)]{BCQC-JIS}:
\[
J_4(n_e^++1)=(1-\delta)(k-1)-\varepsilon+(j-1)=3(1-\delta)-(\varepsilon-(j-1)).
\]
The quantity $t:=3\delta+\varepsilon-(j-1)$ is then the type of $n_e^+$, and it
always lands in $\{0,1,2\}$: the pair $(\varepsilon,j)$ takes nine values, each
forcing $\delta$, and every resulting row of Table~\ref{tab:type-range} yields
$t\in\{0,1,2\}$.

\begin{table}[ht]
\centering
\small
\begin{tabular}{||c|c|c|c||}
\hline\hline
$\varepsilon$ & $j$ & $\delta$ & $t=3\delta+\varepsilon-(j-1)$ \\ \hline
$0$ & $1$ & $0$ & $0$ \\ \hline
$1$ & $1$ & $0$ & $1$ \\ \hline
$2$ & $1$ & $0$ & $2$ \\ \hline
$0$ & $2$ & $1$ & $2$ \\ \hline
$1$ & $2$ & $0$ & $0$ \\ \hline
$2$ & $2$ & $0$ & $1$ \\ \hline
$0$ & $3$ & $1$ & $1$ \\ \hline
$1$ & $3$ & $1$ & $2$ \\ \hline
$2$ & $3$ & $0$ & $0$ \\ \hline\hline
\end{tabular}
\caption{All valid triples $(\varepsilon,j,\delta)$ for $J_4$ and the resulting
type $t$ of $n_e^+$.}
\label{tab:type-range}
\end{table}

\noindent {\rm (c)}
The point $n_e^+$ is fixed exactly when $t=0$. Looking at
Table~\ref{tab:type-range}, $t=0$ happens only in the rows with $\delta=0$ and
$\varepsilon=j-1$ (every row with $\delta=1$ gives $t\in\{1,2\}$). When
$\varepsilon=j-1$ the definition of $\delta$ forces $\delta=0$, and then
\[
n_e^+=\frac{4(n_e+1)-(\varepsilon+1)}{3}=\frac{4(n_e+1)-j}{3}.
\]
Conversely $\varepsilon=j-1$ gives $\delta=0$ and $t=0$, hence $J_4(n_e^+)=n_e^+$;
this is~\cite[Theorem~5(c)]{BCQC-JIS} with $k=4$.
\end{proof}

Table~\ref{tab:type-range} compresses into a single congruence. We record it
here because it answers, in one line each, three questions that recur below:
when a run of pure points of one type continues, when it switches type, and
when it ends at a fixed point.

\begin{corollary}[Successor type]\label{C:successor}
Let $n_e\ge6$ be a high extremal point of $J_4$ of type $t$, and let $n_e^{+}$
be the next high extremal point. Then the type of $n_e^{+}$ is the unique
element of $\{0,1,2\}$ satisfying
\[
\operatorname{type}(n_e^{+})\equiv n_e+t+1 \pmod 3 .
\]
In particular, $n_e^{+}$ has the same type as $n_e$ if and only if
$n_e\equiv2$ (mod $3$), and $n_e^{+}$ is a fixed point if and only if
$n_e\equiv-t-1$ (mod $3$).
\end{corollary}

\begin{proof}
The type-value identity (Lemma~\ref{L:type-value}) gives $j=3-t$, hence
$j-1=2-t$, and Lemma~\ref{L:J4_engine}(b) gives
\[
\operatorname{type}(n_e^{+})=3\delta+\varepsilon-(j-1)
\equiv \varepsilon+t+1\equiv n_e+t+1 \pmod 3 .
\]
The type lies in $\{0,1,2\}$ (Table~\ref{tab:type-range}), so the congruence
determines it. The two special cases are the congruence solved for
$\operatorname{type}(n_e^{+})=t$ and for $\operatorname{type}(n_e^{+})=0$.
\end{proof}

Lemma~\ref{L:J4_engine} does the heavy lifting throughout: from the residue
$\varepsilon$ and the value $j$ it locates the next high extremal point and names its
type, and every later construction is an application of it. The one quantity it asks
us to compute is how far the slope-$4$ segment runs before it would cross the
diagonal, and that is a single floor. We isolate that arithmetic step so we can cite
it cleanly.

\begin{lemma}[Max-floor step for $J_4$]\label{L:maxfloorJ4}
Let $A\ge 2$ be an integer and let $\alpha\in\{1,2,3\}$. Consider the inequality
\begin{equation}\label{eq:maxfloor}
\alpha+4x \le A+1+x,
\end{equation}
where $x$ is a nonnegative integer. Then the largest integer $x$ satisfying
\eqref{eq:maxfloor} is $x_0=\lfloor(A+1-\alpha)/3\rfloor$, and \eqref{eq:maxfloor}
holds precisely for $x\in\{0,1,\dots,x_0\}$.
\end{lemma}

\begin{proof}
Inequality \eqref{eq:maxfloor} is the same as $3x\le A+1-\alpha$. Because $A\ge2$ and
$\alpha\le3$, the right side is at least $0$, so $x=0$ works and the largest
admissible $x$ is $x_0=\lfloor(A+1-\alpha)/3\rfloor$. Every integer between $0$ and
$x_0$ then satisfies the inequality, which gives the solution set $\{0,1,\dots,x_0\}$.
\end{proof}

\begin{remark}[The engine step]\label{R:maxfloor}
Lemmas~\ref{L:J4_engine} and~\ref{L:maxfloorJ4} combine into one move that the
later proofs repeat. Let $n_e\ge6$ be a high extremal point of type $t$, and
set $\alpha:=J_4(n_e+1)=3-t$ (Lemma~\ref{L:type-value}). With
$x_0=\lfloor(n_e+1-\alpha)/3\rfloor$ from Lemma~\ref{L:maxfloorJ4}, the next
high extremal point and the value of $J_4$ there are
\[
n_e^{+}=n_e+1+x_0,
\qquad
J_4(n_e^{+})=\alpha+4x_0,
\]
by Lemma~\ref{L:J4_engine}(a)-(b), and no high extremal point lies strictly
between $n_e$ and $n_e^{+}$. We call this single application the engine step.
\end{remark}

\begin{lemma}[Strict growth and finiteness of the extremal sequence]\label{L:monotone}
Let $n_e\ge6$ be a high extremal point of $J_4$ and let $n_e^{+}$ be the next high
extremal point, as in Lemma~\ref{L:J4_engine}. Then
\begin{equation}\label{eq:growth}
n_e^{+}-n_e\ \ge\ \frac{n_e-2}{3}\ >\ 0 .
\end{equation}
Consequently, the sequence of high extremal points $\ge6$ is strictly increasing,
and between any two consecutive fixed points $n_p^{(\ell)}<n_p^{(\ell+1)}$ there are
only finitely many high extremal points; in particular only finitely many pure
points.
\end{lemma}

\begin{proof}
Set $\varepsilon\equiv n_e$ (mod $3$) with $\varepsilon\in\{0,1,2\}$,
$j:=J_4(n_e+1)\in\{1,2,3\}$, and $\delta:=\delta(\varepsilon,j)\in\{0,1\}$. By
Lemma~\ref{L:J4_engine}(b),
\[
n_e^{+}=\frac{4(n_e+1)-(\varepsilon+1)}{3}-\delta .
\]
Both $\varepsilon$ and $\delta$ are at most their largest values, so
\[
n_e^{+}\ \ge\ \frac{4(n_e+1)-3}{3}-1\ =\ \frac{4n_e-2}{3},
\qquad\text{whence}\qquad
n_e^{+}-n_e\ \ge\ \frac{4n_e-2}{3}-n_e\ =\ \frac{n_e-2}{3}.
\]
For $n_e\ge6$ the bound is strictly positive, so consecutive high extremal
points (from $6$ on) increase. Finiteness follows: the high extremal points
strictly between $n_p^{(\ell)}$ and $n_p^{(\ell+1)}$ are distinct integers in
the finite set $\{n_p^{(\ell)}+1,\dots,n_p^{(\ell+1)}-1\}$, so there are
finitely many of them, and in particular finitely many pure points.
\end{proof}

\begin{remark}[Notation for high extremal points]\label{R:notation-hep}
The symbol $n_e^{+}$ in Lemma~\ref{L:J4_engine} denotes the next high extremal
point following $n_e$; it reappears in Corollary~\ref{C:successor},
Lemma~\ref{L:monotone}, Lemma~\ref{L:segment}, and the proof of
Proposition~\ref{P:term-rule}. For intermediate high extremal points arising in
inductive and block-based arguments, we use the indexed notation
$n_e^{(\ell,q)}$ or $n_e^{(q)}$, depending on context.
\end{remark}

Lemma~\ref{L:monotone} guarantees that the block structure between two fixed
points is finite, which we rely on when assembling the global picture in
Subsection~\ref{sub:main}. We can now describe the first block after a fixed
point. The statement splits by the ternarity of $n_p^{(\ell)}$, and the reason
is structural: the value of $J_4$ just after a fixed point is always $3$ (the
type-value identity with $t=0$), so of the two inputs $(\varepsilon,j)$ the
engine of Lemma~\ref{L:J4_engine} consumes, only the residue
$\varepsilon$ varies, and it alone decides what comes next.

\begin{proposition}[First block after a fixed point]\label{P:initial-block}
Let $n_p^{(\ell)}$ be a fixed point of $J_4$ with $n_p^{(\ell)}\ge 3$.
\begin{enumerate}[label=\arabic*., leftmargin=2.2em]

\item \textbf{Ternarity $2$ (direct transition).}
If $n_p^{(\ell)}\equiv 2$ (mod $3$), then there are no pure points between
$n_p^{(\ell)}$ and $n_p^{(\ell+1)}$, and
\[
n_p^{(\ell+1)}=\frac{4n_p^{(\ell)}+1}{3}.
\]

\item \textbf{Ternarity $1$ (initial block of type-$2$).}
If $n_p^{(\ell)}\equiv 1$ (mod $3$), let $m_2:=\nu_3(4n_p^{(\ell)}+2)$. Then
$m_2\ge1$, the first $m_2$ pure points after $n_p^{(\ell)}$ are all of
type-$2$, the next high extremal point is not, and the last point of this
initial block is
\[
n_e^{(\ell,1)}=\frac{4^{m_2-1}(4n_p^{(\ell)}+2)}{3^{m_2}}-1.
\]

\item \textbf{Ternarity $0$ (initial block of type-$1$).}
If $n_p^{(\ell)}\equiv 0$ (mod $3$), let $m_1:=\nu_3(4n_p^{(\ell)}+3)$. Then
$m_1\ge1$, the first $m_1$ pure points after $n_p^{(\ell)}$ are all of
type-$1$, the next high extremal point is not, and the last point of this
initial block is
\[
n_e^{(\ell,1)}=\frac{4^{m_1-1}(4n_p^{(\ell)}+3)}{3^{m_1}}-1.
\]

\end{enumerate}
\end{proposition}

\begin{proof}
The condition $n_p^{(\ell)}\ge3$ ensures $J_4(n_p^{(\ell)}+1)=3$
(see~\cite[Proposition~2(a)]{BCQC-JIS}), which excludes only $n_p^{(1)}=1$. We
treat the three ternarity classes separately.

\medskip\noindent
\textbf{(1) Ternarity $2$.}
Set $n_e:=n_p^{(\ell)}$. Since $n_e\equiv2$ (mod $3$), we have $\varepsilon=2$ and
$j:=J_4(n_e+1)=3$, giving $\varepsilon=j-1$. By Lemma~\ref{L:J4_engine}(c), the
next high extremal point is a fixed point coinciding with $n_p^{(\ell+1)}$, so there
are no pure points in the interval, and
\[
n_p^{(\ell+1)}=\frac{4(n_e+1)-j}{3}=\frac{4n_p^{(\ell)}+1}{3}.
\]

\medskip\noindent
\textbf{(2) Ternarity $1$.}
Write $n_p^{(\ell)}=3s_1+1$ with $s_1\ge1$. Then $4n_p^{(\ell)}+2=3(4s_1+2)$ is a
multiple of $3$, so $m_2\ge1$.

Start with the first step. With $j=3$, Lemma~\ref{L:J4_engine}(a) gives
$J_4(n_p^{(\ell)}+1+m)=3+4m$ for $m\in\{0,\dots,x_1\}$, where $x_1$ is the largest
integer with $3+4x_1\le n_p^{(\ell)}+1+x_1$. Lemma~\ref{L:maxfloorJ4} (with
$A=n_p^{(\ell)}$, $\alpha=3$) gives
\[
x_1=\left\lfloor\frac{n_p^{(\ell)}-2}{3}\right\rfloor
=\left\lfloor\frac{3s_1-1}{3}\right\rfloor=s_1-1,
\]
so the engine step (Remark~\ref{R:maxfloor}) lands the first high extremal
point after $n_p^{(\ell)}$ at
$n_e^{(1)}=n_p^{(\ell)}+1+x_1=4s_1+1$, with
$J_4(n_e^{(1)})=3+4x_1=4s_1-1=n_e^{(1)}-2$. Thus $n_e^{(1)}$ is of type-$2$.

Now iterate. Put $C_q:=4^{q-1}(4n_p^{(\ell)}+2)/3^q$ for $1\le q\le m_2$; this is a
positive integer, since $3^q$ divides $4n_p^{(\ell)}+2$ for such $q$ while
$\gcd(4^{q-1},3)=1$. We claim, by induction on $q$, that
\begin{equation}\label{eq:ne-type2}
n_e^{(q)}=C_q-1\ \text{ is a type-$2$ pure point},\qquad 1\le q\le m_2.
\end{equation}
The base case $q=1$ is the computation above: $C_1=(4n_p^{(\ell)}+2)/3=4s_1+2$, so
$C_1-1=4s_1+1=n_e^{(1)}$.

Suppose \eqref{eq:ne-type2} holds for some $q<m_2$. Then $3^{q+1}$ divides
$4n_p^{(\ell)}+2$, so $3\mid C_q$; write $C_q=3K$. Because $n_e^{(q)}$ is of type-$2$,
the type-value identity (Lemma~\ref{L:type-value}) gives $J_4(n_e^{(q)}+1)=1$, and
Lemma~\ref{L:maxfloorJ4} (with $A=n_e^{(q)}$, $\alpha=1$) gives
\[
x_{q+1}=\left\lfloor\frac{n_e^{(q)}}{3}\right\rfloor
=\left\lfloor K-\frac{1}{3}\right\rfloor=K-1=\frac{C_q}{3}-1.
\]
The engine step then gives
\[
n_e^{(q+1)}=n_e^{(q)}+1+x_{q+1}
=(C_q-1)+1+\Bigl(\frac{C_q}{3}-1\Bigr)=\frac{4C_q}{3}-1=C_{q+1}-1,
\]
and $J_4(n_e^{(q+1)})=1+4x_{q+1}=4C_q/3-3=n_e^{(q+1)}-2$, so $n_e^{(q+1)}$ is
again of type-$2$. This completes the induction.

The run stops at $q=m_2$, and Corollary~\ref{C:successor} shows why. Each
point of the run satisfies $n_e^{(q)}=C_q-1\equiv C_q+2$ (mod $3$), so the high
extremal point after $n_e^{(q)}$ is again of type-$2$ exactly when
$n_e^{(q)}\equiv2$ (mod $3$), that is, exactly when $3\mid C_q$. For $q<m_2$
this holds, since $3^{q+1}$ divides $4n_p^{(\ell)}+2$; at $q=m_2$ it fails,
since $\nu_3(4n_p^{(\ell)}+2)=m_2$ leaves $3\nmid C_{m_2}$. So the point after
$n_e^{(m_2)}$ is not of type-$2$ and the run ends there. The last type-$2$
point is therefore
\[
n_e^{(\ell,1)}=n_e^{(m_2)}=C_{m_2}-1
=\frac{4^{m_2-1}(4n_p^{(\ell)}+2)}{3^{m_2}}-1.
\]

Table~\ref{T02} records what happens at the second candidate for the three residue
classes of $s_1$ modulo $3$: only $s_1=3s_2+1$ (that is, $3^2\mid(4n_p^{(\ell)}+2)$)
produces another type-$2$ point, while the other two classes give a type-$1$ point or
a fixed point.

\begin{table}[ht]
\centering
\begin{tabular}{||c|c|c|c|c||}
\hline\hline
$s_1$ & $x_2$ & $n_e^{(2)}$ & $J_4(n_e^{(2)})$ & Type of $n_e^{(2)}$ \\ \hline
$3s_2$   & $4s_2$   & $16s_2+2$  & $16s_2+1$  & $1$ \\ \hline
$3s_2+1$ & $4s_2+1$ & $16s_2+7$  & $16s_2+5$  & $2$ \\ \hline
$3s_2+2$ & $4s_2+3$ & $16s_2+13$ & $16s_2+13$ & $0$ \\ \hline\hline
\end{tabular}
\caption{Residue classes of $s_1$ modulo $3$ and the type of the second extremal
candidate after $n_e^{(1)}$. A type-$2$ successor arises only when
$s_1=3s_2+1$, equivalently $3^2\mid(4n_p^{(\ell)}+2)$.}
\label{T02}
\end{table}

\medskip\noindent
\textbf{(3) Ternarity $0$.}
Write $n_p^{(\ell)}=3s_1$ with $s_1\ge1$. Since $4n_p^{(\ell)}+3=3(4s_1+1)$, we
have $3\mid(4n_p^{(\ell)}+3)$, so $m_1\ge1$. The argument runs parallel to
case~(2), with $D_q:=4^{q-1}(4n_p^{(\ell)}+3)/3^q$ in place of $C_q$ and with
$\alpha=2$ in place of $\alpha=1$ at each inductive step, because the
type-value identity gives $J_4(n_e^{(q)}+1)=2$ after a type-$1$ point. We
claim that $n_e^{(q)}=D_q-1$ is a type-$1$ pure point for $1\le q\le m_1$.

The first engine step starts at the fixed point itself, where $j=3$:
Lemma~\ref{L:maxfloorJ4} (with $A=3s_1$, $\alpha=3$) gives
$x_1=\lfloor(3s_1-2)/3\rfloor=s_1-1$, so $n_e^{(1)}=3s_1+1+x_1=4s_1=D_1-1$,
with $J_4(n_e^{(1)})=3+4x_1=4s_1-1=n_e^{(1)}-1$: a type-$1$ point, and indeed
$D_1=(4n_p^{(\ell)}+3)/3=4s_1+1$.

For the inductive step let $q<m_1$, so $3^{q+1}$ divides $4n_p^{(\ell)}+3$ and
$3\mid D_q$; write $D_q=3L$. The engine step from the type-$1$ point
$n_e^{(q)}=3L-1$ uses $\alpha=2$ and gives
\[
x_{q+1}=\left\lfloor\frac{n_e^{(q)}-1}{3}\right\rfloor
=\left\lfloor L-\frac{2}{3}\right\rfloor=L-1,
\qquad
n_e^{(q+1)}=n_e^{(q)}+1+x_{q+1}=4L-1=D_{q+1}-1,
\]
with $J_4(n_e^{(q+1)})=2+4x_{q+1}=4L-2=n_e^{(q+1)}-1$, again of type-$1$.

The run stops at $q=m_1$ for the same reason as in case~(2): each
$n_e^{(q)}\equiv D_q+2$ (mod $3$), so by Corollary~\ref{C:successor} the run
continues past $n_e^{(q)}$ exactly when $3\mid D_q$, and this fails first at
$q=m_1$. The last type-$1$ point is
\[
n_e^{(\ell,1)}=D_{m_1}-1=\frac{4^{m_1-1}(4n_p^{(\ell)}+3)}{3^{m_1}}-1.\qedhere
\]
\end{proof}

\begin{example}[Ternarity $2$: the fixed point $n_p^{(3)}=38$]\label{Ex:ternarity2_38}
Since $38\equiv2$ (mod $3$), Proposition~\ref{P:initial-block}(1) applies: there are no
pure points between $n_p^{(3)}=38$ and $n_p^{(4)}$, and
\[
n_p^{(4)}=\frac{4\cdot38+1}{3}=\frac{153}{3}=51,
\]
in agreement with Table~\ref{tab:J4_fixed_points}.
\end{example}

\begin{example}[Ternarity $1$: the fixed point $n_p^{(11)}=5{,}167$]
\label{Ex:ternarity1_5167}
Since $5{,}167\equiv1$ (mod $3$), Proposition~\ref{P:initial-block}(2) applies. We have
\[
4n_p^{(11)}+2=4\cdot5{,}167+2=20{,}670=3\cdot6{,}890,
\]
and since $6{,}890=2\cdot5\cdot13\cdot53$ is not divisible by $3$, we get
$m_2=\nu_3(20{,}670)=1$. Therefore immediately after $5{,}167$ there is exactly
one pure point of type-$2$, namely
\[
n_e^{(11,1)}=\frac{20{,}670}{3}-1=6{,}890-1=6{,}889.
\]
\end{example}

\begin{example}[Ternarity $0$: the fixed point $n_p^{(15)}=290{,}022$]
\label{Ex:ternarity0_290022}
Since $290{,}022\equiv0$ (mod $3$), Proposition~\ref{P:initial-block}(3) applies. We
have
\[
4n_p^{(15)}+3=4\cdot290{,}022+3=1{,}160{,}091=3^2\cdot128{,}899,
\]
and since the digit sum $1+2+8+8+9+9=37$ is not divisible by $3$, we get
$m_1=\nu_3(1{,}160{,}091)=2$. Therefore immediately after $290{,}022$ there are
exactly two consecutive pure points of type-$1$, with last point
\[
n_e^{(15,1)}=\frac{4\cdot1{,}160{,}091}{9}-1=515{,}596-1=515{,}595.
\]
\end{example}

Proposition~\ref{P:initial-block} settles the first block: from the ternarity of
$n_p^{(\ell)}$ and one $3$-adic valuation we read off its type, its length, and its
last point. It says nothing, though, about what comes after. The terminal point of
the first block is some pure point $n_e$, and to continue we need to know how
many pure points of the opposite type follow it. The next section answers
exactly this question. The mechanism we use looks different in character from
the one above: instead of testing a single fixed expression, we test a moving
family of expressions whose offsets grow with the depth of the block. The
difference turns out to be cosmetic; at the end of the section the moving
family folds back into a single valuation (Corollary~\ref{C:meters-val}), but
the floors are what carry the proof.

\section{Transition meters and local divisibility structure}\label{sec:meters}

Proposition~\ref{P:initial-block} produced the first block after a fixed point. To
continue across the interval we need the rule that takes the last point of one block
and returns the length of the next. That rule is the subject of this section, and it
is purely local: the next block depends only on the last point of the current
one, through a divisibility test on that point. We package the test in two
offset sequences $(A_t)$ and $(B_t)$ and read the block length off the highest
power of~$3$ that passes (Figures~\ref{fig:A-diagram}
and~\ref{fig:B-diagram}). The section ends with a payoff: the whole apparatus
compresses to one $3$-adic valuation per block
(Corollary~\ref{C:meters-val}).

We first name the quantity we are after.

\begin{definition}[Transition meters]\label{D:meters}
For a pure point $n_e$ of $J_4$:
\begin{enumerate}[label=\textbf{(\alph*)}, leftmargin=2.2em]
\item [ {\rm (a)}] if $n_e$ is of type-$1$, let $\mu_{1\to 2}(n_e)$ be the number
of type-$2$ pure points immediately following $n_e$ in the sequence of high
extremal points, that is, the length of the maximal type-$2$ run beginning at
the successor of $n_e$;
\item [ {\rm (b)}] if $n_e$ is of type-$2$, let $\mu_{2\to 1}(n_e)$ be defined in
the same way with the types interchanged.
\end{enumerate}
Both take values in $\mathbb{Z}_{\ge 0}$. A value of $0$ means that no
opposite-type point directly follows $n_e$: the high extremal point after
$n_e$ is then either of the same type, when $n_e$ sits in the middle of its
block, or the next fixed point, when $n_e$ closes the last block of its
interval.
\end{definition}

\begin{remark}[How the test differs from the first block]\label{R:variable-divisibility}
In Proposition~\ref{P:initial-block} the length of the first block came from the
$3$-adic valuation of one fixed number, $4n_p^{(\ell)}+2$ or $4n_p^{(\ell)}+3$. The
test here is different: it is applied not to a single number but to a moving family
$n_e-A_t$ (or $n_e-B_t$), one offset per floor $t$. We climb the floors one at a time,
checking at each floor whether a prescribed power of~$3$ still divides $n_e-A_t$; the
first floor where it fails sets the block length.

The two diagrams, Figures~\ref{fig:A-diagram} and~\ref{fig:B-diagram}, lay this out.
Each floor $t\ge1$ carries two divisibility requirements, shown as the left and right
columns. The $(A_t)$ diagram has in addition a basement, floor~$0$, holding the single
entry condition $3\mid n_e$; the $(B_t)$ diagram has no basement and starts at
floor~$1$. The shaded cell marks the first requirement that fails: a failure in the
left column is the ``odd'' case, one in the right column the ``even'' case, and
the floor at which it happens is the terminal index that the meter formulas
below convert into a block length. The moving family is less alien than it
looks: Corollary~\ref{C:meters-val} at the end of the section folds the whole
tower back into the valuation of a single linear form, recovering the shape of
the first-block test. The floors are how that gets proved.
\end{remark}

\medskip
\noindent
\textbf{Offsets for the type-$1$ to type-$2$ transition.}

For every nonnegative integer $t$, define the offset sequence $(A_t)$ by
\[
A_0 = 0, \qquad
A_t = 9A_{t-1} + 6 \quad (t \ge 1).
\]
Two properties carry all the later arguments: the gap between consecutive
offsets,
\[
A_t - A_{t-1} = 2\cdot 3^{\,2t-1} \qquad (t\ge1),
\]
and the closed form $A_t=\tfrac34(9^t-1)$, obtained by solving the recurrence,
which gives the bound $A_t<3^{2t}$ for every $t\ge0$.

\begin{figure}[ht]
\centering
\renewcommand{\arraystretch}{1.35}

\begin{tabular}{cc}
\textbf{Odd case} & \textbf{Even case} \\[4pt]

\begin{tabular}{|>{\centering\arraybackslash}m{3.2cm}|>{\centering\arraybackslash}m{3.2cm}|}
\hline
\cellcolor{gray!55}$3^{2t}\nmid (n_e-A_t)$ & \\
\hline
$3^{2t-2}\mid (n_e-A_{t-1})$ & $3^{2t-1}\mid (n_e-A_{t-1})$\\
\hline
$\vdots$ & $\vdots$\\
\hline
$3^{2}\mid (n_e-A_{1})$ & $3^{3}\mid (n_e-A_{1})$\\
\hline\hline
\multicolumn{2}{|c|}{$3\mid (n_e-A_{0})$}\\
\hline
\end{tabular}

&

\begin{tabular}{|>{\centering\arraybackslash}m{3.2cm}|>{\centering\arraybackslash}m{3.2cm}|}
\hline
$3^{2t}\mid (n_e-A_t)$ &
\cellcolor{gray!55}$3^{2t+1}\nmid (n_e-A_t)$\\
\hline
$3^{2t-2}\mid (n_e-A_{t-1})$ & $3^{2t-1}\mid (n_e-A_{t-1})$\\
\hline
$\vdots$ & $\vdots$\\
\hline
$3^{2}\mid (n_e-A_{1})$ & $3^{3}\mid (n_e-A_{1})$\\
\hline\hline
\multicolumn{2}{|c|}{$3\mid (n_e-A_{0})$}\\
\hline
\end{tabular}

\end{tabular}

\caption{Divisibility diagram associated with the offsets $(A_t)$ governing the transition from type-$1$ to type-$2$ blocks.}
\label{fig:A-diagram}
\end{figure}

\begin{definition}[Divisibility type with respect to $(A_t)$]\label{D:divA}
Let $n_e$ be a pure point of $J_4$ with $3\mid(n_e-A_0)$, i.e.\ $3\mid n_e$, and let
$t\ge 1$ be an integer. The conditions below are exactly those displayed in
Figure~\ref{fig:A-diagram}; we state them in closed form so that the definition is
self-contained.

\begin{enumerate}
\item [ {\rm (a)}] We say that $n_e$ has odd $A$-divisibility type $t$ if
\[
3^{\,2u}\mid (n_e-A_u)\ \text{ and }\ 3^{\,2u+1}\mid (n_e-A_u)
\quad\text{for all } 0\le u\le t-1,
\qquad\text{but}\qquad
3^{\,2t}\nmid (n_e-A_t).
\]
(The terminal failure occurs in the left-hand column of floor $t$.)

\item [ {\rm (b)}] We say that $n_e$ has even $A$-divisibility type $t$ if
\[
3^{\,2u}\mid (n_e-A_u)\ \text{ and }\ 3^{\,2u+1}\mid (n_e-A_u)
\quad\text{for all } 0\le u\le t-1,
\]
\[
\text{and}\qquad
3^{\,2t}\mid (n_e-A_t),\quad \text{ but }\quad 3^{\,2t+1}\nmid (n_e-A_t).
\]
(The terminal failure occurs in the right-hand column of floor $t$.)
\end{enumerate}

In either case, the integer $t$ is the terminal index of the divisibility
process associated with $(A_t)$; the floor-$0$ basement condition $3\mid(n_e-A_0)$ is
the entry requirement, and successive floors $u\ge1$ each impose the pair of
conditions on $n_e-A_u$ above.
\end{definition}

\medskip
\noindent
\textbf{Offsets for the type-$2$ to type-$1$ transition.}

For every positive integer $t$, define the offset sequence $(B_t)$ by
\[
B_1 = 1, \qquad
B_t = 9B_{t-1} + 10 \quad (t \ge 2).
\]
The same two properties hold here: the gap between consecutive offsets,
\[
B_t - B_{t-1} = 2\cdot 3^{\,2t-2} \qquad (t\ge 2),
\]
and the closed form $B_t=\tfrac14(9^t-5)$, which gives $B_t<3^{2t-1}$ for
every $t\ge1$.

\begin{figure}[ht]
\centering
\renewcommand{\arraystretch}{1.35}

\begin{tabular}{cc}
\textbf{Odd case} & \textbf{Even case} \\[4pt]

\begin{tabular}{|>{\centering\arraybackslash}m{3.2cm}|>{\centering\arraybackslash}m{3.2cm}|}
\hline
\cellcolor{gray!55}$3^{2t-1}\nmid (n_e-B_t)$ & \\
\hline
$3^{2t-3}\mid (n_e-B_{t-1})$ & $3^{2t-2}\mid (n_e-B_{t-1})$\\
\hline
$\vdots$ & $\vdots$\\
\hline
$3\mid (n_e-B_{1})$ & $3^{2}\mid (n_e-B_{1})$\\
\hline
\end{tabular}

&

\begin{tabular}{|>{\centering\arraybackslash}m{3.2cm}|>{\centering\arraybackslash}m{3.2cm}|}
\hline
$3^{2t-1}\mid (n_e-B_t)$ &
\cellcolor{gray!55}$3^{2t}\nmid (n_e-B_t)$\\
\hline
$3^{2t-3}\mid (n_e-B_{t-1})$ & $3^{2t-2}\mid (n_e-B_{t-1})$\\
\hline
$\vdots$ & $\vdots$\\
\hline
$3 \mid (n_e-B_{1})$ & $3^{2}\mid (n_e-B_{1})$\\
\hline
\end{tabular}

\end{tabular}

\caption{Divisibility diagram associated with the offsets $(B_t)$ governing the transition from type-$2$ to type-$1$ blocks.}
\label{fig:B-diagram}
\end{figure}

\begin{definition}[Divisibility type with respect to $(B_t)$]\label{D:divB}
Let $n_e$ be a pure point of $J_4$ and let $t\ge 1$ be an integer. The conditions
below are exactly those displayed in Figure~\ref{fig:B-diagram}; unlike the
$(A_t)$ process there is no basement, and the process begins directly at floor~$1$.

\begin{enumerate}
\item [ {\rm (a)}] We say that $n_e$ has odd $B$-divisibility type $t$ if
\[
3^{\,2u-1}\mid (n_e-B_u)\ \text{ and }\ 3^{\,2u}\mid (n_e-B_u)
\quad\text{for all } 1\le u\le t-1,
\]
\[
\text{but}\qquad
3^{\,2t-1}\nmid (n_e-B_t).
\]
(The terminal failure occurs in the left-hand column of floor $t$. For $t=1$ the
range $1\le u\le 0$ is empty, so the condition reduces to $3\nmid(n_e-B_1)$.)

\item [ {\rm (b)}] We say that $n_e$ has even $B$-divisibility type $t$ if
\[
3^{\,2u-1}\mid (n_e-B_u)\ \text{ and }\ 3^{\,2u}\mid (n_e-B_u)
\quad\text{for all } 1\le u\le t-1,
\]
\[
\text{and}\qquad
3^{\,2t-1}\mid (n_e-B_t),\quad \text{ but }\quad 3^{\,2t}\nmid (n_e-B_t).
\]
(The terminal failure occurs in the right-hand column of floor $t$.)
\end{enumerate}

In either case, the integer $t$ is the terminal index of the divisibility
process associated with $(B_t)$.
\end{definition}

One point deserves emphasis: what these diagrams do and do not do. They do not
generate the
pure points on their own; the points come from $J_4$ itself, through
Lemma~\ref{L:J4_engine}. What the diagram supplies is a finite certificate: starting
from $n_e$, it confirms one opposite-type point after another until a power of~$3$
fails to divide, and the floor of that failure is exactly how far the next block
reaches. The following two lemmas make this precise, building the points floor by
floor.

\begin{remark}[Local notation]\label{R:local-sq}
In Lemmas~\ref{L:finite-recursion} and~\ref{L:finite-recursion-B} below, the symbols
$s_1,s_2,\dots$ denote the auxiliary integers of the floor-by-floor recursion and are
local to those lemmas and their proofs. They are unrelated to the ternary-quotient
integers $s_1,s_2$ used in the proof of Proposition~\ref{P:initial-block}.
\end{remark}

\begin{lemma}[Finite recursive construction up to floor $k$, first case]\label{L:finite-recursion}
Let $n_e$ be a pure point of $J_4$ of type-$1$.
Fix an integer $k\ge 0$ and assume that the $A$-divisibility conditions of
Definition~\ref{D:divA} (equivalently, the odd/even diagram of
Figure~\ref{fig:A-diagram}) hold completely up to floor $k$; that is, all the
required divisibilities are satisfied for floors $0,1,2,\dots,k$. For $k=0$
this is only the basement condition $3\mid n_e$, and the lists below reduce to
the single integer $s_1$ and the single point $n_e^{(1)}$.

Then there exists a sequence of integers
\[
s_1,s_2,\dots,s_{2k+1}
\]
such that
\[
s_1=\frac{n_e}{3}
\]
and the sequence $(s_q)$ satisfies the nested recursion
\begin{equation}\label{eq:sq-rec}
s_q=
\begin{cases}
3s_{q+1}+2, & \text{if $q$ is odd},\\[2pt]
3s_{q+1},   & \text{if $q$ is even},
\end{cases}
\qquad 1\le q\le 2k .
\end{equation}

Moreover, if
\(
n_e^{(1)},n_e^{(2)},\dots,n_e^{(2k+1)}
\)
denote the next $2k+1$ high extremal points following $n_e$, then they satisfy
\begin{equation}\label{eq:neq-rec}
n_e^{(q)}=
\begin{cases}
4^{q}\,s_q+4^{\,q-1}-1, & \text{if $q$ is odd},\\[2pt]
4^{q}\,s_q+3\cdot 4^{\,q-1}-1, & \text{if $q$ is even},
\end{cases}
\qquad 1\le q\le 2k+1,
\end{equation}
and each $n_e^{(q)}$ is a type-$2$ pure point.
\end{lemma}

\begin{proof}
The idea is to build the points one floor at a time, reading off each new term from
the divisibility conditions of Definition~\ref{D:divA} and feeding it through
Lemma~\ref{L:maxfloorJ4}. We check the first floor by hand, then show that if the
conditions hold through floor $k-1$, completing floor $k$ produces exactly two more
terms of the form claimed. Iterating gives the whole list whenever the diagram is
complete up to floor $k$.

\medskip

The basement condition $3\mid(n_e-A_0)=n_e$ lets us write $n_e=3s_1$, so
$s_1=n_e/3$. Since $n_e$ is of type-$1$, the engine step
(Remark~\ref{R:maxfloor}, with $\alpha=2$) gives
$x_0=\lfloor(n_e-1)/3\rfloor=s_1-1$ and the successor
\[
n_e^{(1)}=n_e+1+x_0=4s_1,
\qquad
J_4(n_e^{(1)})=2+4(s_1-1)=n_e^{(1)}-2,
\]
a type-$2$ point, in agreement with \eqref{eq:neq-rec} at $q=1$. This settles
$k=0$.

Now suppose floor~$1$ is complete. Its first column, with $A_1=6$, reads
\[
3^2\mid(n_e-A_1)
\iff 9\mid 3(s_1-2)
\iff 3\mid(s_1-2),
\]
so $s_1=3s_2+2$ for some $s_2\ge0$, which is \eqref{eq:sq-rec} at $q=1$. One more
engine step from the type-$2$ point $n_e^{(1)}=4s_1=12s_2+8$ (with $\alpha=1$,
so $x=\lfloor n_e^{(1)}/3\rfloor=4s_2+2$) gives
\[
n_e^{(2)}=16s_2+11,
\]
matching \eqref{eq:neq-rec} at $q=2$. The second column on floor~$1$ then says
$3^3\mid(n_e-A_1)$; since $n_e-A_1=3(s_1-2)=9s_2$, this is just $3\mid s_2$, so
$s_2=3s_3$ for some $s_3\ge0$, which is \eqref{eq:sq-rec} at $q=2$. A final engine
step from $n_e^{(2)}=48s_3+11$ (again $\alpha=1$, $x=16s_3+3$) yields
\[
n_e^{(3)}=64s_3+15,
\]
as in \eqref{eq:neq-rec} at $q=3$.

\medskip
Assume $k\ge2$ and that the divisibility conditions hold up to floor $k-1$, so that
\begin{equation}\label{eq:last-window-k}
n_e-A_{k-1}=3^{2k-1}s_{2k-1},
\qquad s_{2k-1}\in\mathbb{Z}_{\ge0},
\end{equation}
and that the last constructed extremal point satisfies
\begin{equation}\label{eq:closed-odd-2k1}
n_e^{(2k-1)}=4^{2k-1}s_{2k-1}+4^{2k-2}-1.
\end{equation}
Assume also that floor $k$ is complete.

\smallskip
\noindent(i) Extending the $s$-sequence.
If the first column on floor $k$ holds,
\[
3^{2k}\mid(n_e-A_k),
\]
then using the property $A_k-A_{k-1}=2\cdot3^{2k-1}$ we obtain
\[
n_e-A_k=(n_e-A_{k-1})-(A_k-A_{k-1})
=3^{2k-1}(s_{2k-1}-2).
\]
Hence $3\mid(s_{2k-1}-2)$, so there exists $s_{2k}\ge0$ such that
\begin{equation}\label{eq:s-rec-odd}
s_{2k-1}=3s_{2k}+2,
\end{equation}
which is precisely the odd recursion in \eqref{eq:sq-rec}.

\smallskip
\noindent (ii) Computing $x_1$.
The point $n_e^{(2k-1)}$ is of type-$2$, so the type-value identity
(Lemma~\ref{L:type-value}) gives $J_4(n_e^{(2k-1)}+1)=1$, and
Lemma~\ref{L:maxfloorJ4} (with $A=n_e^{(2k-1)}$, $\alpha=1$) gives
\[
x_1=\Bigl\lfloor\frac{n_e^{(2k-1)}}{3}\Bigr\rfloor.
\]
Using \eqref{eq:closed-odd-2k1} and \eqref{eq:s-rec-odd},
\begin{align}
\frac{n_e^{(2k-1)}}{3}
&=\frac{4^{2k-1}(3s_{2k}+2)+4^{2k-2}-1}{3}\\
&=4^{2k-1}s_{2k}+3\cdot4^{2k-2}-\frac13.
\end{align}
Thus
\begin{equation}\label{eq:X-value}
x_1=4^{2k-1}s_{2k}+3\cdot4^{2k-2}-1.
\end{equation}

\smallskip
\noindent (iii) Constructing $n_e^{(2k)}$.
The engine step (Remark~\ref{R:maxfloor}) gives
\[
n_e^{(2k)}=n_e^{(2k-1)}+1+x_1.
\]
Substituting the previous expressions yields
\[
n_e^{(2k)}=4^{2k}s_{2k}+3\cdot4^{2k-1}-1,
\]
which matches the even case of \eqref{eq:neq-rec}.

\smallskip
\noindent (iv) Type verification.
By the same engine step,
\[
J_4(n_e^{(2k)})=1+4x_1
=4^{2k}s_{2k}+3\cdot4^{2k-1}-3
=n_e^{(2k)}-2,
\]
so $n_e^{(2k)}$ is of type-$2$.

\smallskip
\noindent (v) Completing the second column.
If the divisibility condition in the second column on floor $k$ also holds,
\[
3^{2k+1}\mid(n_e-A_k),
\]
then since $n_e-A_k=3^{2k-1}(s_{2k-1}-2)=3^{2k}s_{2k}$ by step~(i) and
\eqref{eq:s-rec-odd}, we obtain $3\mid s_{2k}$, hence
\begin{equation}\label{eq:s-rec-even}
s_{2k}=3s_{2k+1},
\end{equation}
which is the even recursion in \eqref{eq:sq-rec}.

\smallskip
\noindent (vi) Computing $x_2$.
By step~(iv) the point $n_e^{(2k)}$ is of type-$2$, so again
$J_4(n_e^{(2k)}+1)=1$ and Lemma~\ref{L:maxfloorJ4} (with $A=n_e^{(2k)}$,
$\alpha=1$) gives
\[
x_2=\Bigl\lfloor\frac{n_e^{(2k)}}{3}\Bigr\rfloor.
\]
Using \eqref{eq:s-rec-even},
\[
x_2=4^{2k}s_{2k+1}+4^{2k-1}-1.
\]

\smallskip
\noindent (vii) Constructing $n_e^{(2k+1)}$.
One more engine step,
\[
n_e^{(2k+1)}=n_e^{(2k)}+1+x_2
=4^{2k+1}s_{2k+1}+4^{2k}-1,
\]
which matches the odd case of \eqref{eq:neq-rec}.

\smallskip
\noindent (viii) Type verification.
Again,
\[
J_4(n_e^{(2k+1)})=1+4x_2
=4^{2k+1}s_{2k+1}+4^{2k}-3
=n_e^{(2k+1)}-2,
\]
so $n_e^{(2k+1)}$ is also of type-$2$.

Finally, step~(v) gives $n_e-A_k=3^{2k}s_{2k}=3^{2k+1}s_{2k+1}$, which is
\eqref{eq:last-window-k} with $k$ replaced by $k+1$; together with
\eqref{eq:closed-odd-2k1} at index $2k+1$, established in step~(vii), this is
exactly the inductive hypothesis for the next floor. This completes the finite
recursive construction up to floor $k$ and the proof of the lemma.
\end{proof}

\begin{lemma}[Finite recursive construction up to floor $k$, second case]\label{L:finite-recursion-B}
Let $n_e$ be a pure point of $J_4$ of type-$2$.
Fix an integer $k\ge 1$ and assume that the $B$-divisibility conditions of
Definition~\ref{D:divB} (equivalently, the odd/even diagram of
Figure~\ref{fig:B-diagram}) hold completely up to floor $k$; that is, all the required
divisibilities are satisfied for floors $1,2,\dots,k$.

Then there exists a sequence of integers
\[
s_1,s_2,\dots,s_{2k}
\]
such that
\[
s_1=\frac{n_e-1}{3}
\]
and the sequence $(s_q)$ satisfies the nested recursion
\[
s_q=
\begin{cases}
3s_{q+1}, & \text{if $q$ is odd},\\[2pt]
3s_{q+1}+2, & \text{if $q$ is even},
\end{cases}
\qquad 1\le q\le 2k-1 .
\]

Moreover, if
\[
n_e^{(1)},n_e^{(2)},\dots,n_e^{(2k)}
\]
denote the next $2k$ high extremal points following $n_e$, then they satisfy
\[
n_e^{(q)}=
\begin{cases}
4^{q}\,s_q+3\cdot 4^{\,q-1}-1, & \text{if $q$ is odd},\\[2pt]
4^{q}\,s_q+4^{\,q-1}-1, & \text{if $q$ is even},
\end{cases}
\qquad 1\le q\le 2k,
\]
and each $n_e^{(q)}$ is a type-$1$ pure point.
\end{lemma}

\begin{proof}
The argument parallels that of Lemma~\ref{L:finite-recursion}, with three
structural differences forced by the offset family $(B_t)$: the diagram has no
basement (it begins at floor~$1$, with $B_1=1$ rather than $A_0=0$); the seed is
$s_1=(n_e-1)/3$ rather than $n_e/3$; and the additive constants $3\cdot4^{\,q-1}$
and $4^{\,q-1}$ in the closed form are interchanged between the odd and even
parities. We give the base case and one inductive step in full, the points at which
these differences enter.

\medskip
Since $n_e$ is of type-$2$, the type-value identity (Lemma~\ref{L:type-value}) gives $J_4(n_e+1)=1$,
so the first step uses $j=1$. The entry condition of the $B$-diagram is its
floor-$1$ left column $3\mid(n_e-B_1)$, that is, $3\mid(n_e-1)$, which holds
exactly when $n_e\equiv1$ (mod $3$). Write $n_e=3s_1+1$, so $s_1=(n_e-1)/3$. With $j=1$,
Lemma~\ref{L:maxfloorJ4} gives $x_1=\lfloor n_e/3\rfloor=s_1$, and
Lemma~\ref{L:J4_engine} produces
\[
n_e^{(1)}=n_e+1+s_1=4s_1+2=4^{1}s_1+3\cdot4^{0}-1,
\]
the odd case ($q=1$) of the closed form. Its value is
$J_4(n_e^{(1)})=j+4x_1=1+4s_1=n_e^{(1)}-1$, so $n_e^{(1)}$ is of type-$1$.

\medskip
Suppose the construction has reached the type-$1$ point $n_e^{(q)}$ given by the
closed form, and that the divisibility conditions allow it to continue. There are two
sub-steps, according to the parity of $q$; both apply Lemma~\ref{L:J4_engine} to a
type-$1$ point, where $j=2$ and the floor count is
$x=\lfloor(n_e^{(q)}-1)/3\rfloor$.

From odd $q$ to even $q+1$: write $q=2u-1$, so
$n_e^{(2u-1)}=4^{2u-1}s_{2u-1}+3\cdot4^{2u-2}-1$. The left column of floor $u$ forces
the odd recursion $s_{2u-1}=3s_{2u}$, and substituting gives
$n_e^{(2u-1)}=12\cdot4^{2u-2}s_{2u}+3\cdot4^{2u-2}-1$. Then
\[
x=\Bigl\lfloor\frac{n_e^{(2u-1)}-1}{3}\Bigr\rfloor=4^{2u-1}s_{2u}+4^{2u-2}-1,
\]
so
\[
n_e^{(2u)}=n_e^{(2u-1)}+1+x=4^{2u}s_{2u}+4^{2u-1}-1,
\qquad J_4(n_e^{(2u)})=2+4x=n_e^{(2u)}-1,
\]
the even case of the closed form, again of type-$1$.

From even $q$ to odd $q+1$: write $q=2u$, so
$n_e^{(2u)}=4^{2u}s_{2u}+4^{2u-1}-1$. The right column of floor $u$ forces the even
recursion $s_{2u}=3s_{2u+1}+2$, and substituting gives
$n_e^{(2u)}=3\cdot4^{2u}s_{2u+1}+2\cdot4^{2u}+4^{2u-1}-1$. Then
\[
x=\Bigl\lfloor\frac{n_e^{(2u)}-1}{3}\Bigr\rfloor=4^{2u}s_{2u+1}+3\cdot4^{2u-1}-1,
\]
so
\[
n_e^{(2u+1)}=n_e^{(2u)}+1+x=4^{2u+1}s_{2u+1}+3\cdot4^{2u}-1,
\qquad J_4(n_e^{(2u+1)})=n_e^{(2u+1)}-1,
\]
the odd case of the closed form, of type-$1$. Each completed floor thus contributes
two more type-$1$ points, exactly as in Lemma~\ref{L:finite-recursion}, and the
construction runs as long as the floors stay complete. This proves the lemma.
\end{proof}

The next proposition turns these constructions into a formula: the length of each
block equals the floor at which its divisibility process stops. Nothing like this is
needed for $J_3$~\cite{BCQC-J3}, where the pure points between fixed points behave more
simply and no moving divisibility test arises; the two pure-point types of $J_4$ are
what force the floor-by-floor mechanism here.

Both transitions are proved below; the first sets the pattern, and the second
reuses it through Lemma~\ref{L:finite-recursion-B}.

\begin{proposition}[Transition meters via divisibility types]\label{P:meters-div}
Let $n_e$ be a pure point of $J_4$ lying in an interval
$(n_p^{(\ell)},n_p^{(\ell+1)})$ with $n_p^{(\ell)}\ge3$. When $3\mid n_e$
(respectively $3\mid(n_e-B_1)$), the point $n_e$ has a well-defined odd or
even $A$-divisibility (respectively $B$-divisibility) type; this is Step~0
of the proof.

\begin{enumerate}[label=\arabic*., leftmargin=2.2em]

\item \textbf{Type-$1$ to type-$2$ transition.}
Assume that $n_e$ is of type-$1$.

\begin{enumerate}[label=\alph*., leftmargin=2.0em]

\item If $3\nmid (n_e-A_0)$, then
\[
\mu_{1\to 2}(n_e)=0.
\]

\item Assume that $t\ge1$. Then the transition meter of $n_e$ is given by
\[
\mu_{1\to 2}(n_e)
=
\begin{cases}
2t-1, & \text{if $n_e$ has odd $A$-divisibility type $t$},\\[4pt]
2t,   & \text{if $n_e$ has even $A$-divisibility type $t$}.
\end{cases}
\]
\end{enumerate}

\item \textbf{Type-$2$ to type-$1$ transition.}
Assume that $n_e$ is of type-$2$.

\begin{enumerate}[label=\alph*., leftmargin=2.0em]

\item If $3\nmid (n_e-B_1)$, then
\[
\mu_{2\to 1}(n_e)=0.
\]

\item Assume that $t\ge1$ and that $n_e$ has $B$-divisibility type $t$. Then the
transition meter of $n_e$ is given by
\[
\mu_{2\to 1}(n_e)
=
\begin{cases}
2t-2, & \text{if $n_e$ has odd $B$-divisibility type $t$},\\[4pt]
2t-1,   & \text{if $n_e$ has even $B$-divisibility type $t$}.
\end{cases}
\]
In particular, an odd $B$-divisibility type $t=1$ (the failure
$3\nmid(n_e-B_1)$ occurring already at the first column of floor~$1$) gives
$\mu_{2\to1}(n_e)=0$, recovering case~(a).
\end{enumerate}

\end{enumerate}
\end{proposition}

\begin{proof}
Step 0: the divisibility processes terminate.
Suppose first that $n_e\ne A_t$ for every $t$, and pick $t$ with $A_t>n_e$.
The closed form gives $0<A_t-n_e\le A_t<3^{2t}$, so $3^{2t}\nmid(n_e-A_t)$ and
the $A$-process fails at floor $t$ at the latest. If instead $n_e=A_{t_0}$ for
some $t_0$, then the gap identity gives
$n_e-A_{t_0+1}=-2\cdot3^{\,2t_0+1}$, and $3^{2t_0+2}\nmid 2\cdot3^{\,2t_0+1}$,
so the process fails at floor $t_0+1$. The $B$-process terminates by the same
two-case argument, using $B_t<3^{2t-1}$ and the gap
$B_t-B_{t-1}=2\cdot3^{\,2t-2}$. The first failing condition fixes the terminal
index and its parity, so the divisibility types of Definitions~\ref{D:divA}
and~\ref{D:divB} are well defined.

\medskip
\noindent Part 1: type-$1$ to type-$2$.
Corollary~\ref{C:successor} supplies the continuation criterion used in both
parts: the high extremal point after a pure point $m$ has the same type as $m$
if and only if $m\equiv2$ (mod $3$).

\smallskip
\noindent Case (a).
Here $3\nmid n_e$. By Corollary~\ref{C:successor} the point after the
type-$1$ point $n_e$ has type $\equiv n_e+2$ (mod $3$); type-$2$ would force
$n_e\equiv0$ (mod $3$), which is excluded. So no type-$2$ point follows $n_e$,
and $\mu_{1\to2}(n_e)=0$.

\smallskip
\noindent Case (b).
Let $n_e$ have $A$-divisibility type $t\ge1$, so the conditions hold through
floor $t-1$ and the first failure occurs on floor $t$.
Lemma~\ref{L:finite-recursion}, applied with $k=t-1$, produces the integers
$s_1,\dots,s_{2t-1}$ and the consecutive type-$2$ points
$n_e^{(1)},\dots,n_e^{(2t-1)}$ following $n_e$, with
\[
n_e-A_{t-1}=3^{\,2t-1}s_{2t-1},
\qquad
n_e^{(2t-1)}=4^{2t-1}s_{2t-1}+4^{2t-2}-1 ;
\]
for $t=1$ this is the basement case $k=0$, with $s_1=n_e/3$ and
$n_e^{(1)}=4s_1$. Since $4\equiv1$ (mod $3$), the closed form gives
$n_e^{(2t-1)}\equiv s_{2t-1}$ (mod $3$), and the gap identity
$A_t-A_{t-1}=2\cdot3^{\,2t-1}$ turns the floor-$t$ left column into the
continuation criterion at $n_e^{(2t-1)}$:
\begin{align}
3^{2t}\mid(n_e-A_t)
&\iff 3^{2t}\mid 3^{\,2t-1}(s_{2t-1}-2)
\iff s_{2t-1}\equiv2 \pmod 3 \nonumber\\
&\iff n_e^{(2t-1)}\equiv2 \pmod 3 . \label{eq:contA}
\end{align}

If the type is odd, the left column fails, so $n_e^{(2t-1)}\not\equiv2$
(mod $3$), and by the continuation criterion the point after $n_e^{(2t-1)}$ is
not of type-$2$. The run consists of exactly the points
$n_e^{(1)},\dots,n_e^{(2t-1)}$, and $\mu_{1\to2}(n_e)=2t-1$.

If the type is even, the left column holds, so $s_{2t-1}=3s_{2t}+2$ for some
integer $s_{2t}\ge0$, and one more engine step is available. The computation
in the inductive step of Lemma~\ref{L:finite-recursion} (with $k=t$) applies
verbatim and yields the type-$2$ point
\[
n_e^{(2t)}=4^{2t}s_{2t}+3\cdot4^{2t-1}-1\equiv s_{2t}+2 \pmod 3 ,
\]
together with $n_e-A_t=3^{2t}s_{2t}$. The right column then reads
\[
3^{2t+1}\mid(n_e-A_t)\iff 3\mid s_{2t}\iff n_e^{(2t)}\equiv2 \pmod 3 ,
\]
and its failure stops the run after $n_e^{(2t)}$ by the continuation
criterion. Hence $\mu_{1\to2}(n_e)=2t$.

\medskip
\noindent Part 2: type-$2$ to type-$1$.

\smallskip
\noindent Case (a).
By Corollary~\ref{C:successor} the point after the type-$2$ point $n_e$ has
type $\equiv n_e+3\equiv n_e$ (mod $3$); type-$1$ would force $n_e\equiv1$
(mod $3$), that is, $3\mid(n_e-B_1)$, which is excluded. So
$\mu_{2\to1}(n_e)=0$.

\smallskip
\noindent Case (b) with $t=1$.
Here $3\mid(n_e-B_1)$; write $n_e=3s_1+1$. The engine step from $n_e$ (with
$\alpha=1$) gives $x_0=\lfloor n_e/3\rfloor=s_1$ and the type-$1$ point
\[
n_e^{(1)}=n_e+1+s_1=4s_1+2\equiv s_1+2 \pmod 3 ,
\qquad
J_4(n_e^{(1)})=1+4s_1=n_e^{(1)}-1 .
\]
The floor-$1$ right column $3^{2}\mid(n_e-B_1)=3s_1$ says $3\mid s_1$, that
is, $n_e^{(1)}\equiv2$ (mod $3$). For even type $t=1$ this fails, so the run
stops after the single point $n_e^{(1)}$ and $\mu_{2\to1}(n_e)=1=2t-1$. Odd
type $t=1$ is case~(a) and gives $0=2t-2$.

\smallskip
\noindent Case (b) with $t\ge2$.
Lemma~\ref{L:finite-recursion-B}, applied with $k=t-1$, produces
$s_1,\dots,s_{2t-2}$ and the consecutive type-$1$ points
$n_e^{(1)},\dots,n_e^{(2t-2)}$ following $n_e$, with
\[
n_e-B_{t-1}=3^{\,2t-2}s_{2t-2},
\qquad
n_e^{(2t-2)}=4^{2t-2}s_{2t-2}+4^{2t-3}-1\equiv s_{2t-2} \pmod 3 .
\]
The gap identity $B_t-B_{t-1}=2\cdot3^{\,2t-2}$ turns the floor-$t$ left
column into the continuation criterion at $n_e^{(2t-2)}$:
\[
3^{2t-1}\mid(n_e-B_t)
\iff 3^{2t-1}\mid 3^{\,2t-2}(s_{2t-2}-2)
\iff n_e^{(2t-2)}\equiv2 \pmod 3 .
\]
If the type is odd, this fails and the run stops at $n_e^{(2t-2)}$, giving
$\mu_{2\to1}(n_e)=2t-2$. If the type is even, write $s_{2t-2}=3s_{2t-1}+2$;
the engine step from the type-$1$ point $n_e^{(2t-2)}$ (with $\alpha=2$), as
computed in the inductive step of Lemma~\ref{L:finite-recursion-B} from even
$q=2t-2$ to odd $q=2t-1$, yields the type-$1$ point
\[
n_e^{(2t-1)}=4^{2t-1}s_{2t-1}+3\cdot4^{2t-2}-1\equiv s_{2t-1}+2 \pmod 3 ,
\]
together with $n_e-B_t=3^{\,2t-1}s_{2t-1}$. The right column reads
\[
3^{2t}\mid(n_e-B_t)\iff 3\mid s_{2t-1}\iff n_e^{(2t-1)}\equiv2 \pmod 3 ,
\]
and its failure stops the run after $n_e^{(2t-1)}$. Hence
$\mu_{2\to1}(n_e)=2t-1$.
\end{proof}
Which half of Proposition~\ref{P:meters-div} applies is decided by the type of the
current point: a type-$1$ point starts a $1\to2$ transition, a type-$2$ point a
$2\to1$ transition. Running the two halves alternately walks across an entire
interval, as the next example shows.

\begin{example}[Alternating use of the transition meters]\label{Ex:altmeters_6889}
Recall from Example~\ref{Ex:ternarity1_5167} that the fixed point $n_p^{(11)}=5{,}167$
opens with a type-$2$ block whose last point is $n_e^{(11,1)}=6{,}889$. Starting there,
we apply Proposition~\ref{P:meters-div} repeatedly, each time using the half that
matches the type of the current point.

\medskip\noindent Step 1 (type $2\to1$).
For $n_e=6{,}889$, $\;n_e-B_1=6{,}888$, and $3\mid6{,}888$ but $3^2\nmid6{,}888$, so
$n_e$ has even $B$-divisibility type $t=1$ and $\mu_{2\to1}(n_e)=2t-1=1$. The next
block is one type-$1$ point, $n_e^{(11,2)}=9{,}186$.

\medskip\noindent Step 2 (type $1\to2$).
For $n_e=9{,}186$, $\;n_e-A_0=9{,}186$ with $3\mid n_e$, and $3^3\mid(n_e-A_1)$ but
$3^4\nmid(n_e-A_2)$, so $n_e$ has odd $A$-divisibility type $t=2$ and
$\mu_{1\to2}(n_e)=2t-1=3$. The next block is three type-$2$ points, $12{,}248$,
$16{,}331$, $21{,}775$, ending at $n_e^{(11,3)}=21{,}775$.

\medskip\noindent Step 3 (type $2\to1$).
For $n_e=21{,}775$, $\;n_e-B_1=21{,}774$, with $3\mid21{,}774$ but $3^2\nmid21{,}774$:
even $B$-divisibility type $t=1$, so $\mu_{2\to1}(n_e)=1$ and the next block is one
type-$1$ point, $n_e^{(11,4)}=29{,}034$.

\medskip\noindent Step 4 (type $1\to2$).
For $n_e=29{,}034$, $\;n_e-A_0=29{,}034$ with $3\mid n_e$ but $3^2\nmid(n_e-A_1)$: odd
$A$-divisibility type $t=1$, so $\mu_{1\to2}(n_e)=1$ and the next block is one
type-$2$ point, $n_e^{(11,5)}=38{,}712$.

\medskip\noindent Stopping.
For $n_e=38{,}712$, $\;n_e-B_1=38{,}711$ and $3\nmid38{,}711$: odd $B$-divisibility
type $t=1$, so $\mu_{2\to1}(n_e)=0$ and the alternation ends.

\medskip
The interval therefore carries the five blocks below, the same ones displayed in
Example~\ref{Ex:pattern21212}:
\begin{table}[ht]
\centering
\begin{tabular}{c|rrrrr}
 & $\mathcal B_1$ & $\mathcal B_2$ & $\mathcal B_3$ & $\mathcal B_4$ & $\mathcal B_5$ \\ \hline
type-$2$ & $6{,}889$ &           & $12{,}248;\,16{,}331;\,21{,}775$ &           & $38{,}712$ \\
type-$1$ &           & $9{,}186$ &                                  & $29{,}034$ &            \\
\end{tabular}
\caption{The five blocks in $(5{,}167;\,51{,}617)$, listed by the last point of each
block (the middle type-$2$ block has three points).}
\label{tab:altmeters6889}
\end{table}
\end{example}

The two divisibility processes look like genuinely new machinery, with their
moving offsets and paired columns, but they compress completely. Each offset
sits one step away from an exact power of $3$, and that turns the whole
diagram into a single valuation.

\begin{corollary}[Valuation form of the transition meters]\label{C:meters-val}
Let $n_e$ be a pure point of $J_4$ as in Proposition~\ref{P:meters-div}. Then
\begin{align}
\mu_{1\to 2}(n_e)&=\nu_3(4n_e+3) \qquad \text{if $n_e$ is of type-$1$},\label{eq:mu12val}\\
\mu_{2\to 1}(n_e)&=\nu_3(4n_e+5) \qquad \text{if $n_e$ is of type-$2$}.\label{eq:mu21val}
\end{align}
\end{corollary}

\begin{proof}
The closed forms of the offsets produce two exact powers of $3$:
\[
4A_t+3=3(9^t-1)+3=3^{2t+1} \quad (t\ge0),
\qquad
4B_t+5=(9^t-5)+5=3^{2t} \quad (t\ge1).
\]
Since $\gcd(4,3)=1$, multiplying by $4$ preserves divisibility by any power of
$3$, so for $0\le f\le 2t+1$,
\[
3^{f}\mid(n_e-A_t)
\iff 3^{f}\mid 4(n_e-A_t)=(4n_e+3)-3^{2t+1}
\iff 3^{f}\mid(4n_e+3),
\]
and in the same way, for $0\le f\le 2t$,
$3^{f}\mid(n_e-B_t)\iff 3^{f}\mid(4n_e+5)$. Every condition in
Definitions~\ref{D:divA} and~\ref{D:divB} has exactly this shape, with the
exponent never exceeding that of the subtracted power, so the whole diagram
for $n_e$ reads off a single integer. For the $A$-process: the basement fails
precisely when $\nu_3(4n_e+3)=0$, since $4n_e+3\equiv n_e$ (mod $3$); the
conditions persist through floor $t-1$ with the left column of floor $t$
failing precisely when $\nu_3(4n_e+3)=2t-1$; and the right column of floor $t$
fails on top of the left one precisely when $\nu_3(4n_e+3)=2t$. Comparing
with Proposition~\ref{P:meters-div}(1) gives \eqref{eq:mu12val} in every
case. The $B$-process is identical with $4n_e+5$: odd type $t$ is
$\nu_3(4n_e+5)=2t-2$, even type $t$ is $\nu_3(4n_e+5)=2t-1$, and the failure
$3\nmid(n_e-B_1)$ is $\nu_3(4n_e+5)=0$; comparing with
Proposition~\ref{P:meters-div}(2) gives \eqref{eq:mu21val}.
\end{proof}

\begin{remark}[One linear form per block]\label{R:meters-val}
Corollary~\ref{C:meters-val} puts the meters in the same form as the
first-block counts $m_1=\nu_3(4n_p^{(\ell)}+3)$ and
$m_2=\nu_3(4n_p^{(\ell)}+2)$ of Proposition~\ref{P:initial-block}: every block
length in an interval is the $3$-adic valuation of a linear form $4n+c$ with
$c\in\{2,3,5\}$, evaluated at the point that precedes the block. The five
meter evaluations of Example~\ref{Ex:altmeters_6889} compress to
\[
\nu_3(27{,}561)=1;\quad
\nu_3(36{,}747)=3;\quad
\nu_3(87{,}105)=1;\quad
\nu_3(116{,}139)=1;\quad
\nu_3(154{,}853)=0,
\] after using \eqref{eq:mu12val}-\eqref{eq:mu21val}, respectively.
The floor diagrams remain the proof mechanism: the valuation says how long
each block is, while the floor-by-floor construction of
Lemmas~\ref{L:finite-recursion} and~\ref{L:finite-recursion-B} is what
certifies, point by point, that the block is really there.
\end{remark}

\section{The transition structure and its use}\label{sec:structure}

The previous sections produced the local rules one at a time: the first block after a
fixed point, the length of each block from the one before it, and the local linearity
of $J_4$. This section closes the loop. We first supply the termination rule that turns
the last pure point of an interval into the next fixed point
(Subsection~\ref{sub:arrival}); we then assemble all the rules into a single structural
theorem (Subsection~\ref{sub:main}); and finally we put that structure to work,
evaluating $J_4(n)$ at an arbitrary $n$ without the linear-time recursion
(Subsection~\ref{sub:evaluation}).

\subsection{Arrival at the next fixed point}\label{sub:arrival}

The example of Section~\ref{sec:meters} stops once a transition meter returns $0$,
leaving a last pure point in hand: $n_e^{(11,5)}=38{,}712$. One piece is still missing.
We know that no further pure point follows, so the next high extremal point must be the
fixed point $n_p^{(\ell+1)}$ itself, but we have not yet said how to compute it from the
last pure point. This is what the termination rule supplies: once the
last point and its type are known, the next fixed point is one application of the engine
away.

\begin{proposition}[Termination rule and computation of the next fixed point]\label{P:term-rule}
Let $n_e^{(\ell,r)}$ be the last high extremal point between consecutive fixed points
$n_p^{(\ell)}$ and $n_p^{(\ell+1)}$.
\begin{enumerate}[label=\arabic*. , leftmargin=2.2em]
\item If $n_e^{(\ell,r)}$ is of type-$2$, then
\[
n_p^{(\ell+1)}=\frac{4n_e^{(\ell,r)}}{3}+1.
\]
\item If $n_e^{(\ell,r)}$ is of type-$1$, then
\[
n_p^{(\ell+1)}=\frac{4n_e^{(\ell,r)}+2}{3}.
\]
\end{enumerate}
\end{proposition}

\begin{proof}
Let $n_e:=n_e^{(\ell,r)}$ be the last high extremal point in
$(n_p^{(\ell)};\,n_p^{(\ell+1)})$. Since $n_e\ge 6$, Lemma~\ref{L:J4_engine} applies;
put
\[
j:=J_4(n_e+1)\in\{1, 2, 3\}, \qquad 
\varepsilon\equiv n_e\pmod 3,\ \ \varepsilon\in\{0, 1, 2\}.
\]
Because $n_e$ is the last high extremal point inside the interval, and no high
extremal point lies strictly between $n_e$ and its successor $n_e^{+}$
(Remark~\ref{R:maxfloor}), the successor is the fixed point $n_p^{(\ell+1)}$
itself. By Lemma~\ref{L:J4_engine}(c) this forces
$\varepsilon=j-1$, in which case
\begin{equation}\label{eq:nextfixed_from_engine}
n_p^{(\ell+1)}=n_e^{+}=\frac{4(n_e+1)-j}{3}.
\end{equation}
The two cases differ only in the value of $j$.

\medskip
\noindent Case 1: $n_e$ of type-$2$.
By the type-value identity (Lemma~\ref{L:type-value}), $j=J_4(n_e+1)=1$, so
$\varepsilon=0$ and $n_e\equiv0$ (mod $3$). With $j=1$,
\eqref{eq:nextfixed_from_engine} becomes
\[
n_p^{(\ell+1)}=\frac{4(n_e+1)-1}{3}=\frac{4n_e}{3}+1.
\]

\medskip
\noindent Case 2: $n_e$ of type-$1$.
Here Lemma~\ref{L:type-value} gives $j=J_4(n_e+1)=2$, so $\varepsilon=1$ and
$n_e\equiv1$ (mod $3$). With $j=2$,
\[
n_p^{(\ell+1)}=\frac{4(n_e+1)-2}{3}=\frac{4n_e+2}{3}.
\]
This gives both formulas.
\end{proof}

\begin{example}[Closing the interval $(5{,}167;\,51{,}617)$]
\label{Ex:termination_38712}
Example~\ref{Ex:altmeters_6889} ended at the last pure point $n_e^{(11,5)}=38{,}712$,
of type-$2$. Proposition~\ref{P:term-rule}(1) then closes the interval:
\[
n_p^{(12)}=\frac{4\cdot38{,}712}{3}+1=51{,}617,
\]
matching Table~\ref{tab:J4_fixed_points}. So the passage from $n_p^{(11)}=5{,}167$ to
$n_p^{(12)}=51{,}617$ runs through five blocks and ends with one division.
\end{example}

With this, the local picture is complete: we can find the first block, the length of
every block after it, and the fixed point at which the interval closes. What remains
is to state the combined result.

\subsection{The main structural theorem}\label{sub:main}

We can now collect the pieces. Proposition~\ref{P:initial-block} fixes the first
block, Proposition~\ref{P:meters-div} the length of each subsequent one, and
Proposition~\ref{P:term-rule} the exit to the next fixed point; the theorem below
states what their combination yields.

\begin{theorem}[Fixed-point transition structure for $J_4$]\label{thm:main}
Let $n_p^{(\ell)}$ and $n_p^{(\ell+1)}$ be consecutive fixed points of $J_4$ with
$n_p^{(\ell)}\ge3$ (that is, $\ell\ge2$, excluding only $n_p^{(1)}=1$). Exactly one of
the following holds.

\medskip
\noindent\textbf{(i)} The interval
$(n_p^{(\ell)},n_p^{(\ell+1)})$ contains no pure points.

\medskip
\noindent\textbf{(ii)} The pure points of the interval split into a
finite list of blocks
\[
\mathcal B_1;\,\mathcal B_2;\,\dots;\,\mathcal B_r;\,\qquad r=r(\ell)\ge1,
\]
each a maximal run of pure points of one type, with consecutive blocks of opposite
types.

\medskip
In case (ii) the whole transition is determined, with no reference to intermediate
values of $J_4$, by four ingredients:
\begin{enumerate}[label=\textup{(\arabic*)}, leftmargin=2.4em]
\item the ternarity of $n_p^{(\ell)}$, together with the valuation
$m_1=\nu_3(4n_p^{(\ell)}+3)$ or $m_2=\nu_3(4n_p^{(\ell)}+2)$, which fix the type,
length, and last point of $\mathcal B_1$ (Proposition~\ref{P:initial-block});
\item the transition meters $\mu_{1\to2}$ and $\mu_{2\to1}$, which give the length of
each later block from the last point of the previous one
(Proposition~\ref{P:meters-div}; in valuation form, Corollary~\ref{C:meters-val});
\item the first meter value $0$, which marks the last block; and
\item the termination rule, which turns the last pure point into $n_p^{(\ell+1)}$
(Proposition~\ref{P:term-rule}).
\end{enumerate}
\end{theorem}

\begin{proof}
If the transition is direct, the interval contains no pure points and there is nothing
more to prove, so assume it is not.

By Proposition~\ref{P:initial-block}, the ternarity of $n_p^{(\ell)}$ and the
valuation $m_1$ or $m_2$ produce the first block $\mathcal B_1$ and its last point.
Given the last point of any block, Proposition~\ref{P:meters-div} returns the next
block through the matching meter $\mu_{1\to2}$ or $\mu_{2\to1}$, a run of pure points
of the opposite type. Repeating this generates blocks of alternating type.

The iteration cannot run forever. Each block consists of high extremal points $\ge6$,
and by Lemma~\ref{L:monotone} every meter step moves to a strictly larger point; since
all of them lie in the finite interval $(n_p^{(\ell)},n_p^{(\ell+1)})$, only finitely
many occur, giving a finite list $\mathcal B_1,\dots,\mathcal B_r$. The process halts at the first meter value~$0$ (Proposition~\ref{P:meters-div},
cases~1(a) and~2(a)). Let $m$ be the last pure point at that moment. Its
successor under the engine is a high extremal point
(Lemma~\ref{L:J4_engine}(b)); it is not of the opposite type, because the
meter is~$0$; and it is not of the same type, because the meter propositions
produce maximal runs, so the block containing $m$ ends at $m$. (Equivalently,
$m\not\equiv2$ (mod $3$), the continuation criterion of
Corollary~\ref{C:successor}.) The only type left is~$0$: the successor is the
fixed point $n_p^{(\ell+1)}$, and the termination rule
(Proposition~\ref{P:term-rule}) computes it from $m$. Each step used only the
ingredients listed, so the whole passage is determined by them.
\end{proof}

In short, the pure points between two fixed points are not scattered arbitrarily:
they fall into a finite alternating run of blocks whose lengths are read off from
explicit $3$-adic valuations. The theorem is also constructive. Starting from a fixed
point, one builds the first block from its ternarity, generates the remaining blocks
by iterating the transition meters, and recovers the next fixed point from the
termination rule.

\subsection{Evaluation of $J_4$ at an arbitrary argument}\label{sub:evaluation}

The structural description above has a direct computational
payoff: it yields a formula for $J_4(n)$ at an arbitrary $n$ that avoids the
linear-time recursion $J_4(n)=\bigl(J_4(n-1)+3\bigr)\bmod n+1$. The idea is that
$J_4$ is piecewise linear of slope~$4$ between consecutive high extremal points, so
once we locate the high extremal point that opens the segment containing $n$, the
value $J_4(n)$ is given by a single linear expression.

We first record the segment formula, which is just a restatement of
Lemma~\ref{L:J4_engine}(a).

\begin{lemma}[Linear segments of $J_4$]\label{L:segment}
Let $n_e\ge6$ be a high extremal point of $J_4$, let $j:=J_4(n_e+1)$, and let
$n_e^{+}$ be the next high extremal point. Then
\begin{equation}\label{eq:segment}
J_4(n)=j+4\,(n-n_e-1)\qquad\text{for every } n\in\{n_e+1,\dots,n_e^{+}\}.
\end{equation}
\end{lemma}

\begin{proof}
By Lemma~\ref{L:J4_engine}(a), $J_4(n_e+1+m)=j+4m$ for $0\le m\le m_0$, and by
Lemma~\ref{L:J4_engine}(b) the next high extremal point is $n_e^{+}=n_e+1+m_0$.
Setting $m=n-n_e-1$ gives \eqref{eq:segment} on the stated range.
\end{proof}

To use \eqref{eq:segment} we must find, for a given $n$, the high extremal point
$n_e$ that opens its segment. The high extremal points of a single interval are few
and are generated one from the next by the engine, so this is a short walk. Write
$\Phi(n_e)$ for the map sending a high extremal point to its successor: with
$j=3-t$ the value of $J_4$ just after a point of type $t$,
$\varepsilon\in\{0,1,2\}$ the residue of $n_e$
(mod $3$), and $\delta=1$ if $\varepsilon<j-1$ and $\delta=0$ otherwise,
Lemma~\ref{L:J4_engine}(b) gives
\begin{equation}\label{eq:Phi}
\Phi(n_e)=\frac{4(n_e+1)-(\varepsilon+1)}{3}-\delta,
\qquad
\operatorname{type}\bigl(\Phi(n_e)\bigr)=3\delta+\varepsilon-(j-1).
\end{equation}

\begin{theorem}[Evaluation of $J_4$ via the fixed point sequence]\label{thm:evaluation}
Let $n\ge7$, and let $n_p^{(\ell)}$ be the largest fixed point with
$n_p^{(\ell)}\le n$. Define a finite sequence of high extremal points by
\[
e_0:=n_p^{(\ell)},\qquad e_{i+1}:=\Phi(e_i),
\]
and let $r$ be the first index with $n\le e_{r+1}$. Then
\[
J_4(n)=j_r+4\,(n-e_r-1),
\qquad j_r:=3-\operatorname{type}(e_r),
\]
where $\operatorname{type}(e_0)=0$ and the types of $e_1,e_2,\dots$ are computed
along the way from \eqref{eq:Phi}. If $n=n_p^{(\ell)}$ then $J_4(n)=n$.
\end{theorem}

\begin{proof}
By the definition of $\Phi$ as the next-high-extremal-point map of
Lemma~\ref{L:J4_engine}(b), the points $e_0,e_1,\dots$ are exactly the high extremal
points of $J_4$ from $n_p^{(\ell)}$ onward, listed in increasing order. Suppose first
that $n$ is not a fixed point. Then $n_p^{(\ell)}<n<n_p^{(\ell+1)}$, since
$n_p^{(\ell)}$ is the largest fixed point at most $n$, so $n$ falls in one of the
half-open segments $(e_i,e_{i+1}]$, which tile $(e_0,\infty)$ because each
$e_{i+1}$ is the high extremal point immediately after $e_i$. The index $r$
singled out in the statement satisfies $e_r<n$: if $r=0$ this is
$n>n_p^{(\ell)}$, and if $r\ge1$ then $n>e_r$ because $r$ is the first index
with $n\le e_{r+1}$. Lemma~\ref{L:segment}, applied at the high extremal point
$e_r$ with $j_r=J_4(e_r+1)=3-\operatorname{type}(e_r)$ (Lemma~\ref{L:type-value}),
gives $J_4(n)=j_r+4(n-e_r-1)$. The walk is finite
because $\Phi$ strictly increases (Lemma~\ref{L:monotone}) and only finitely many high
extremal points lie below $n$. Finally, if $n=n_p^{(\ell)}$ then $n$ is a fixed point
and $J_4(n)=n$ by definition.
\end{proof}

The cost of this evaluation is the cost of locating $n_p^{(\ell)}$ in the fixed point
sequence plus the length of the walk. The walk visits $r+1$ high extremal points,
which is the number of high extremal points of $J_4$ in $[\,n_p^{(\ell)},n\,]$; this
is at most the total number of pure points in the interval
$(n_p^{(\ell)},n_p^{(\ell+1)})$ plus one, a quantity that does not grow with $n$ inside
the interval. Each step of the walk is a constant number of arithmetic operations on
integers of $O(\log n)$ bits. So the procedure never touches the
$\Theta\bigl(n_p^{(\ell+1)}-n_p^{(\ell)}\bigr)$ intermediate arguments that the
naive recursion $J_4(n)=(J_4(n-1)+3)\bmod n+1$ steps through; locating the
segment, not scanning to it, is what makes the evaluation fast. Extending the
fixed point table itself, when $n$ outruns it, costs little more: by
Corollary~\ref{C:meters-val}, the recurrence of Theorem~\ref{thm:main} spends
one $3$-adic valuation per block and one exact division per interval.

\begin{example}[Evaluating $J_4(10{,}000)$]\label{Ex:eval10000}
The largest fixed point not exceeding $10{,}000$ is $n_p^{(11)}=5{,}167$ (see
Table~\ref{tab:J4_fixed_points}), so we set $e_0=5{,}167$, of type-$0$. The map
$\Phi$ gives the successive high extremal points
\[
e_0=5{,}167\ \xrightarrow{\ \Phi\ }\ e_1=6{,}889\ (\text{type }2)
\ \xrightarrow{\ \Phi\ }\ e_2=9{,}186\ (\text{type }1)
\ \xrightarrow{\ \Phi\ }\ e_3=12{,}248 .
\]
Since $9{,}186<10{,}000\le 12{,}248$, the segment containing $10{,}000$ opens at
$e_2=9{,}186$, which is of type-$1$, so $j_2=3-1=2$. By
Theorem~\ref{thm:evaluation},
\[
J_4(10{,}000)=2+4\,(10{,}000-9{,}186-1)=2+4\cdot813=3{,}254,
\]
reached after a walk of only three high extremal points rather than the several
thousand steps of the naive recursion.
\end{example}

\section{Concluding remarks}\label{sec:conclusion}

This paper has established a complete recursive description of the transition between
consecutive fixed points of the Josephus function $J_4$. The main structural result,
Theorem~\ref{thm:main}, shows that the pure high extremal points between any two
consecutive fixed points $n_p^{(\ell)}$ and $n_p^{(\ell+1)}$ organize into a finite
alternating run of type-$1$ and type-$2$ blocks, and that
the entire transition is governed by explicitly computable ingredients: the
ternarity of $n_p^{(\ell)}$, the first-block valuations $\nu_3(4n_p^{(\ell)}+2)$
and $\nu_3(4n_p^{(\ell)}+3)$, the transition meters $\mu_{1\to2}$ and
$\mu_{2\to1}$, and the termination rule of Proposition~\ref{P:term-rule}. Every
block length in the chain is a single $3$-adic valuation of a linear form
$4n+c$ with $c\in\{2,3,5\}$ (Corollary~\ref{C:meters-val}), evaluated at the
point that precedes the block, so the whole interval is read off by a short
list of valuations.

Compared with the case $k=3$ of~\cite{BCQC-J3}, the case $k=4$ is more
intricate, though less than it first appears. For $k=3$ a single divisibility
index $\overline{m}_\ell=\nu_2(3n_p^{(\ell)}+2)$ carries one fixed point to the
next. For $k=4$ a single index no longer suffices: the two types of pure points
alternate, so an interval needs a chain of valuations rather than one. Each
link in that chain is itself a single $3$-adic valuation
(Corollary~\ref{C:meters-val}); the two offset sequences $(A_t)$ and $(B_t)$
and the paired floor tests are the machinery that proves it, not extra data the
user must compute. The genuine new feature at $k=4$ is therefore the
alternation, not the valuations, and one expects the alternation to grow with
$k$.

Several directions remain open. The most immediate is $k=5$. The pattern of the
solved cases suggests that the high extremal points there fall into $k-2=3$ pure
types, falling $1$, $2$, or $3$ short of the diagonal, with the blocks between
consecutive fixed points alternating among these types; we have not proved
this, and even the number of types is part of what a treatment of $k=5$ would
have to settle. Numerical experiments up to $2\cdot10^6$ are consistent with
the picture: exactly the three predicted types occur, and consecutive blocks
carry distinct types. The divisibility-based method of this paper should still apply, but with
more offset sequences and a more involved interaction among them.

More broadly, one may ask for a single framework covering all $k\ge2$, in which
the step from one fixed point to the next is governed by finitely many valuations
of explicit linear forms in the current fixed point. So far the relevant modulus
is $k-1$: base $2$ governs $k=3$ and base $3$ governs $k=4$, while $k=2$ is
degenerate ($k-1=1$, fixed points $2^\ell-1$). For $k=3$ and $k=4$ the modulus
$k-1$ is prime, so these are genuine $p$-adic valuations; the next value, $k=5$,
is the first with $k-1$ composite, and whether the right object there is a
$4$-adic count or a pair of $2$-adic ones is exactly the question that
distinguishes a clean general framework from a case-by-case ascent. We have not
settled it.

The fixed point recurrence also gives an efficient way to evaluate $J_4(n)$ at an
arbitrary argument. Subsection~\ref{sub:evaluation} makes this precise:
Theorem~\ref{thm:evaluation} locates the discrete linear segment of $J_4$ containing $n$ by a
short walk through the high extremal points of a single fixed point interval, and
then reads off $J_4(n)$ from one linear expression, never visiting the intervening
arguments. This is the $k=4$ analogue of the direct evaluation available for $k=3$
through~\cite{BCQC-J3}. It would be of interest to compress the walk itself into a
closed form, expressing the relevant high extremal point directly in terms of
$n_p^{(\ell)}$ and the block data, so that $J_4(n)$ is obtained without any
iteration at all.

\section{Acknowledgments}

The first author was partially supported by the National Science Foundation
under grant DMS-2307328.

\bigskip
\hrule
\bigskip

\noindent 2020 {\it Mathematics Subject Classification}:
Primary 11B37; Secondary 11A07, 11B83, 05A99.

\noindent \emph{Keywords:}
Josephus problem, Josephus function, fixed point, $3$-adic valuation, recurrence,
high extremal point, integer sequence.

\bigskip
\hrule
\bigskip

\noindent (Concerned with sequences
\seqnum{A006257}, \seqnum{A032434}, \seqnum{A182459},
and \seqnum{A385333}.)

\end{document}